\documentclass[12pt,a4paper]{article}
\usepackage[english]{babel}

\usepackage{pstricks}
\usepackage{pst-node}
\usepackage{amsmath,boxedminipage}
\usepackage{amssymb}
\usepackage{graphicx}
\usepackage{enumerate}
\usepackage{color}
\usepackage[text={15cm,24cm}]{geometry}
\usepackage[normalem]{ulem}
\usepackage[ansinew]{inputenc}
\usepackage{ulem}
\usepackage{float}
\usepackage{subcaption}
\usepackage{diagbox}
\usepackage{booktabs}

\usepackage{thmtools}
\usepackage{thm-restate}
\usepackage{hyperref}
\usepackage{cleveref}

\newenvironment{proof}{{\bf Proof}:\ }%
   {~\ \hfill $\Box$\vspace{0,5cm}}

    {~\ \hfill$\Box$\vspace{0,5cm}}
\newenvironment{ack}{\vskip5mm{\bf Acknowledgements:}}%

\newtheorem{theorem}{Theorem}[section]
\newtheorem{prop}[theorem]{Property}
\newtheorem{rmk}[theorem]{Remark}

\newtheorem{lemma}[theorem]{Lemma}
\newtheorem{proposition}[theorem]{Proposition}

\newtheorem{coro}[theorem]{Corollary}

\newtheorem{claim}{Claim}

\newcounter{claimish}[theorem]

\graphicspath{{.}{graphics/}}
\newrgbcolor{lightlightlightgray}{0.9 0.9 0.9}

\numberwithin{equation}{section}

\begin{document}
\title{The complexity of the bondage problem \\ in planar graphs}

\author{V. Bouquet\footnotemark[1]}
\date{\today}

\def\thefootnote{\fnsymbol{footnote}}

\footnotetext[1]{ \noindent
Conservatoire National des Arts et M\'etiers, CEDRIC laboratory, Paris (France). Email: {\tt
valentin.bouquet@cnam.fr}}

\maketitle

\begin{abstract}
   A set $S\subseteq V(G)$ of a graph $G$ is a dominating set if each vertex has a neighbor in $S$ or belongs to $S$. Let $\gamma(G)$ be the cardinality of a minimum dominating set in $G$. The bondage number $b(G)$ of a graph $G$ is the smallest cardinality of a set of edges $A\subseteq E(G)$, such that $\gamma(G-A)=\gamma(G)+1$. The $d$-\textsc{Bondage} is the problem of deciding, given a graph $G$ and an integer $d\geq 1$, if $b(G)\leq d$. This problem is known to be $\mathsf{NP}$-hard even for bipartite graphs and $d=1$.

   In this paper, we show that $1$-\textsc{Bondage} is $\mathsf{NP}$-hard, even for the class of $3$-regular planar graphs, the class of subcubic claw-free graphs, and the class of bipartite planar graphs of maximum degree $3$, with girth $k$, for any fixed $k\geq 3$. On the positive side, for any planar graph $G$ of girth at least $8$, we show that we can find, in polynomial time, a set of three edges $A$ such that $\gamma(G-A)>\gamma(G)$.
   Last, we exposed some classes of graphs for which \textsc{Dominating Set} can be solved in polynomial time, and where $d$-\textsc{Bondage} can also be solved in polynomial time, for any fixed $d\geq 1$.

\vspace{0.2cm}
\end{abstract}

\section{Introduction}\label{intro}

Given a graph $G=(V,E)$, a set $S\subseteq V$ is called a dominating set if every vertex is an element of $S$ or is adjacent to an element of $S$. The minimum cardinality of a dominating set in $G$ is called the domination number and is denoted by $\gamma(G)$. A dominating set $S\subseteq V$ with $\vert S\vert=\gamma(G)$ is called a minimum dominating set, for short a $\gamma$-set. \textsc{Dominating Set} is the problem of deciding, given a graph $G$ and an integer $k\geq 1$, if $\gamma(G)\leq k$. It is a well-known $\mathsf{NP}$-complete problem. For an overview of the topics in graph domination, we refer to the book of Haynes et al.\ \cite{Haynes}. The bondage number has been introduced by Fink et al.\ \cite{Fink} has a parameter to measure the criticality of a graph, in respect to the domination number. The bondage number $b(G)$ of a graph $G$ is the minimum number of edges whose removal from $G$ increases the domination number, that is, with edges $E'\subseteq E$ such that $\gamma(G-E') = \gamma(G)+1$. We say that an edge $e$ is $\gamma$-critical if its removal increase the dominating number, that is, $\gamma(G-e)=\gamma(G)+1$. Therefore a graph $G$ has a $\gamma$-critical edge if and only if $b(G)=1$. The \textsc{Bondage} problem is defined as follows:

\begin{center}
   \begin{boxedminipage}{.99\textwidth}
   $d$-\textsc{Bondage} \\[2pt]
   \begin{tabular}{ r p{0.8\textwidth}}
   \textit{~~~~Instance:} &a graph $G=(V,E)$.\\
   \textit{Question:} &is $b(G)\leq d$ ?
   \end{tabular}
   \end{boxedminipage}
\end{center}

The $1$-\textsc{Bondage} problem has been shown to be $\mathsf{NP}$-hard in \cite{HuXu}, and it has been shown in \cite{HuSohn} that it remains $\mathsf{NP}$-hard when restricted to bipartite graphs. In this paper, we strengthen this result as follows:

\begin{theorem}
   For any fixed $k\geq 3$, $1$-\textsc{Bondage} is $\mathsf{NP}$-hard for bipartite planar graphs with maximum degree $3$ and girth at least $k$.
\end{theorem}

We think that this result is of interest because of the following upper bound proved by Fischermann et al.\ in \cite{FischRau}.

\begin{theorem}{\cite{FischRau}}
   Let $G$ be a planar graph of girth at least $8$. Then $b(G)\leq 3$.
\end{theorem}

Fortunately we were able to extend this result to prove the following:

\begin{proposition}
   Let $G=(V,E)$ be a planar graph of girth at least $8$. Then we can find a set $E'\subseteq E$, where $\vert E'\vert = 3$, such that $\gamma(G-E')>\gamma(G)$ in polynomial time.
\end{proposition}

For an extending overview of the bondage number and its related properties, we refer to the survey of Xu \cite{Xu}.

\section{Notations and preliminaries}

The graphs considered in this paper are finite and simple, that is, without directed edges or loops or parallel edges. The reader is referred to \cite{Bondy} for definitions and notations in graph theory, and to \cite{GareyJohnson} for definitions and terminology concerning complexity theory. \smallskip

Let $G=(V,E)$ be a graph with vertex set $V=V(G)$ and edge set $E=E(G)$. Let $v\in V$ and $xy\in E$. We say that $x$ and $y$ are the \textit{endpoints} of the edge. The \textit{degree} of $v$ in $G$ is $d_G(v)$ or simply $d(v)$ when the referred graph is obvious. When $d(v)=0$ we say that $v$ is \textit{isolated}, and when $d(v)=1$ we say that $v$ is a \textit{leaf}. Let $\delta(G)$ and $\Delta(G)$ denote its minimum degree and its maximum degree, respectively. A \textit{$k$-vertex} is a vertex of degree $k$. The graph $G$ is \textit{$k$-regular} whenever it contains only $k$-vertices. We say that a graph is \textit{cubic} if it is $3$-regular, and is \textit{subcubic} if $\Delta(G)\leq 3$. We denote by $N_G(v)$ the \textit{open neighborhood} of a vertex $v$ in $G$, and $N_G[v]=N_G(v)\cup\{v\}$ its \textit{closed neighborhood} in $G$. When it is clear from context, we note $N(v)$ and $N[v]$. For a set $U\subseteq V$, its \textit{open neighborhood} is $N(U)=\{N(u)\setminus U \mid u\in U\}$, and its \textit{closed neighborhood} is $N[U]=\{N[u] \mid u\in U\}$. For a subset $U\subseteq V$, let $G[U]$ denote the subgraph of $G$ induced by $U$, which has vertex set $U$ and edge set $\{uv\in E \mid u,v\in U\}$. We may refer to $U$ as an \textit{induced subgraph} of $G$ when it is clear from the context. If a graph $G$ has no induced subgraph isomorphic to a fixed graph $H$, we say that $G$ is \textit{$H$-free}. Let $g(G)$ denote the \textit{girth} of $G$, that is, the minimum length of an induced cycle. By convention, $g(G)=+\infty$ when $G$ is acyclic. \smallskip

For a vertex $v\in V$, we write $G-v=(V\setminus \{v\}, E\setminus E')$, where $E'=\{uv \mid u\in N(v)\}$, and for a subset $V'\subseteq V$, we write $G-V'=(V\setminus V',E\setminus E')$, where $E'=\{uv \mid v\in V', u\in N(v)\}$. For an edge $e\in E$, we write $G-e=(V,E\setminus \{e\})$, and for a subset $E'\subseteq E$, we write $G-E'=(V, E\setminus E')$. For $n\geq 1$, the graph $P_n=u_1-u_2-\cdots-u_n$ denotes the \textit{cordless path} or \textit{induced path} on $n$ vertices, that is, $V({P_n})=\{u_1,\ldots,u_n\}$, and $E({P_n})=\{u_iu_{i+1}\; |\; 1\leq i\leq n-1\}$. For a path $P_n$ we say that $P_n$ is an \textit{even path} if $n$ is even, else it is an \textit{odd path}. We denote by $d(u,v)$ the \textit{distance} between two vertices, that is, the length of a shortest path between $u$ and $v$. Note that when $uv\in E$, $d(u,v)=1$. For $n\geq 3$, the graph $C_n=u_1-u_2-\cdots-u_n-u_1$ denotes the \textit{cordless cycle} or \textit{induced cycle} on $n$ vertices, that is, $V({C_n})=\{u_1,\ldots,u_n\}$ and $E({C_n})=\{u_iu_{i+1}\; |\; 1\leq i\leq n-1\}\cup \{u_nu_1\}$. For $n\ge 4$, $C_n$ is called a \textit{hole}. The \textit{triangle} is $C_3$. The \textit{claw} $K_{1,3}$ is the 4-vertex \textit{star}, that is, the graph with vertices $u$, $v_1$, $v_2$, $v_3$ and edges $uv_1$, $uv_2$, $uv_3$. We say that $u$ is the center of the claw. The \textit{paw} is the graph with four vertices $u_1,u_2,u_3$, $v_1$ and edges $u_1u_2,u_2u_3,u_1u_3,u_1v_1$. A class of graphs $\mathcal{C}$ is \textit{closed under edge deletions} when for every graph $G\in \mathcal{C}$, and for every $e\in E(G)$, we have $G-e\in \mathcal C$. \smallskip

For two vertex disjoint induced subgraphs $A,B$ of $G$: $A$ is \textit{complete} to $B$ if $ab$ is an edge for any $a\in V(A)$ and $b\in V(B)$, and $A$ is \textit{anti-complete} to $B$ if $ab$ is not an edge for any $a\in V(A)$ and $b\in V(B)$. A set $U\subseteq V$ is called a \textit{clique} if any pairwise distinct vertices $u,v\in U$ are adjacent. The graph $K_n$ is the clique with $n$ vertices. When $G$ is a clique, then $G$ is a \textit{complete graph}. A set $U\subseteq V$ is called a \textit{stable set} or an \textit{independent set} if any pairwise distinct vertices $u,v\in U$ are non adjacent. A set $U\subseteq V$ is called a \textit{vertex cover} if every edge $e$ of $G$ has an endpoint in $U$. We denote by $\tau(G)$ the size of a minimum vertex cover of $G$ and by $\alpha(G)$ the size of a maximum independent set of $G$. A $\tau$-set of $G$ is a minimum vertex cover of $G$ and an $\alpha$-set of $G$ is a maximum independent set of $G$. We denote by \textit{$\alpha$-core$(G)$} and \textit{$\tau$-core$(G)$} the set of vertices that belongs to every $\alpha$-set and every $\tau$-set, respectively. The set of vertices that belongs to no $\alpha$-set and no $\tau$-set is called \textit{$\alpha$-anticore$(G)$}, \textit{$\tau$-anticore$(G)$}, respectively. We say that a graph $G$ has \textit{$X$-cores} or \textit{$X$-anticores} when $X$-core$(G)\neq \emptyset$ or $X$-anticore$(G)\neq \emptyset$, where $X\in \{\alpha,\gamma,\tau\}$. Note that given an $\alpha$-set $S$ of $G$, the set of vertices $V\setminus S$ is a $\tau$-set of $G$, and vice-versa. Hence $\alpha$-core$(G)=\tau$-anticore$(G)$ and $\alpha$-anticore$(G)=\tau$-core$(G)$.

\section{Cores and anticores}

Before giving hardness results for the Bondage problem in planar graphs, we introduce the two following problems:

\begin{center}
   \begin{boxedminipage}{.99\textwidth}
   \textsc{\sc $\alpha$-Core} \\[2pt]
   \begin{tabular}{ r p{0.8\textwidth}}
   \textit{~~~~Instance:} &A graph $G=(V,E)$. \\
   \textit{Question:} & Is there a vertex that belongs to every \textsc{Maximum Independent Set} of $G$, that is, $\vert \alpha$-core$(G)\vert \geq 1$\,?
   \end{tabular}
   \end{boxedminipage}
\end{center}

\begin{center}
   \begin{boxedminipage}{.99\textwidth}
   \textsc{\sc $\tau$-Anticore} \\[2pt]
   \begin{tabular}{ r p{0.8\textwidth}}
   \textit{~~~~Instance:} &A graph $G=(V,E)$. \\
   \textit{Question:} & Is there a vertex that belongs to no \textsc{Minimum Vertex Cover} of $G$\, that is, $\vert \tau$-core$(G)\vert \geq 1$\,?
   \end{tabular}
   \end{boxedminipage}
\end{center}


Boros et al.\ studied $\alpha$-cores in \cite{Boros}. They \textit{announced} that for any fixed integer $k\geq 0$, deciding if a given graph $G$ is such that $\vert \alpha$-core$(G)\vert >k$ is $\mathsf{NP}$-complete. In fact, they only showed that this problem is $\mathsf{NP}$-hard. Note that $\alpha$-\textsc{Core} is the same problem for $k=1$. We remark that if $\alpha$-\textsc{Core} is in NP, then there must exists a certificate, that we can test in polynomial time, with the following characteristics: given a graph $G$ and $C\subseteq V(G)$, is $C\in \alpha$-core$(G)$? Therefore we should be able to test if a given vertex is in $\alpha$-core$(G)$ in polynomial time. Yet, we had the intuition that this would not be possible, unless $\mathsf{P}=\mathsf{NP}$. The following result proved us right.

\begin{theorem}\label{No-Poly-Core-MIS}
   Given a graph $G$ and a vertex $v$, there is no polynomial time algorithm deciding if $v\in \alpha$-core$(G)$, unless $\mathsf{P}=\mathsf{NP}$.
\end{theorem}
\begin{proof}
   We define as \textsc{Core-Mis} the problem of deciding, given a graph $G=(V,E)$ and a vertex $v\in V$, if $v\in \alpha$-core$(G)$. We remark that if $v\in \alpha$-core$(G)$, then $\alpha(G-v)=\alpha(G)-1$, else $\alpha(G)=\alpha(G-v)$. We show that if \textsc{Core-Mis} can be solved in polynomial time, then \textsc{Maximum Independent Set} (\textsc{Mis} for short) can also be solved in polynomial time.

   Let $(G,k)$ be an instance for \textsc{Mis} and let $v$ be a vertex of $G$. We solve \textsc{Core-Mis} with input $(G,v)$ in polynomial time. If $v\in \alpha-$core$(G)$, then $G$ has an independent set of size $k$ if and only if $G-v$ has an independent set of size $k-1$. Hence the instance $(G-v,l-1)$ is equivalent to $G(G,l)$ for \textsc{Mis}. If $v\not\in \alpha$-core$(G)$, then $G$ has an independent set of size $k$ if and only if $G-v$ has an independent set of size $k$. Hence the instance $(G-v,l)$ is equivalent to $(G,l)$ for \textsc{Mis}.

   Therefore, we can use the polynomial time algorithm for \textsc{Core-Mis}, to have an equivalent instance of \textsc{Mis} with one less vertex. Hence by running this algorithm iteratively on the vertices of $G$, we can solve \textsc{Mis} in polynomial time. Yet \textsc{Mis} is known to be $\mathsf{NP}$-complete, see \cite{GareyJohnson}. This completes the proof.
\end{proof}

From Theorem \ref{No-Poly-Core-MIS} it follows that $\alpha$-\textsc{Core} is not in $\mathsf{NP}$. From the result of Boros et al in \cite{Boros}, it follows:

\begin{theorem}{\cite{Boros}}\label{NPH_alpha_core_anti}
   $\alpha$-\textsc{Core} is $\mathsf{NP}$-hard.
\end{theorem}

Since $\alpha$-core$(G)=\tau$-anticore$(G)$ for any graph $G$, we obtain the following corollary.

\begin{coro}\label{NPH_tau_core_anti}
   $\tau$-\textsc{Anticore} is $\mathsf{NP}$-hard.
\end{coro}

We describe one operation that will be of use to prove that $\tau$-anticore remains $\mathsf{NP}$-hard for planar graphs. Given a graph $G=(V,E)$ and $v\in V(G)$, we define the graph $G_v+u=(V',E')$, where $u\not\in V(G)$ as follows: $V'=V\cup \{u\}$ and $E'=E\cup \{uv\}$.

\begin{prop}\label{tau_anticore}
   Let $v$ be a vertex of a graph $G=(V,E)$, then $v \in \tau$-anticore$(G)$ if and only if $\tau(G_v+u)= \tau(G)+1$.
\end{prop}
\begin{proof}
   Let $v$ be a vertex of a graph $G$. Any $\tau$-set of $G_v+u$ contains either $u$ or $v$ to cover the edge $uv$. Therefore if $v\in \tau$-anticore$(G)$, then $\tau(G_v+u)=\tau(G)+1$. Else there exists a $\tau$-set $S$ of $G$ containing $v$ that is also a $\tau$-set of $G_v+u$.
\end{proof}

We show that $\tau$-\textsc{Anticore} remains $\mathsf{NP}$-hard for planar graphs.
\begin{theorem}\label{NPH_tau_anticore_planar}
   $\tau$-\textsc{Anticore} is $\mathsf{NP}$-hard when restricted to planar graphs.
\end{theorem}
\begin{proof}
   Garey, Johnson and Stockmeyer in \cite{Garey} proved that \textsc{Vertex Cover} (referred to as \textsc{Node Cover}) is $\mathsf{NP}$-complete when restricted to planar graphs. To that end, they gave a polynomial reduction from \textsc{Vertex Cover} with no restrictions on the graph. We claim that the graph $G'=(V',E')$ constructed from $G=(V,E)$ has the following property: $\tau$-anticore$(G')\neq \emptyset$ if and only if $\tau$-anticore$(G)\neq \emptyset$. First we give an overview of their proof. \\

   Given a graph $G=(V,E)$, we construct a graph $G'=(V',E')$. Embed $G$ in the plane, allowing edges to cross but so that no more than two edges meet at any point, and no edge cross a vertex other than its own endpoint. Then replace each crossing by a copy of $H$ as shown in Figure \ref{crossing_edges} to build $G'$. For each $i,j$, where $0\leq i,j\leq 2$, let $c[i,j]$ be the cardinality of a minimal vertex cover $C$ of $H$ such that $\vert \{v_1,v_1'\}\cap C\vert = i$ and $\vert \{v_2,v_2'\}\cap C\vert = j$. Note that $v_1,v_1',v_2,v_2'$ are the four corners of $H$, see Figure \ref{planar_H}. The following properties hold:

   \begin{table}[H]
      \centering
      \begin{tabular}{l r r r}
         \toprule
         \diagbox[linewidth=.05em,outerrightsep=1pt]{j}{i} & 0 & 1 & 2 \\ [0.5ex]
         \midrule
         0 & 13 & 14 & 15 \\
         1 & 13 & 13 & 14 \\
         2 & 14 & 14 & 15 \\
         \bottomrule
      \end{tabular}
      \caption{The values of $c[i,j]$.}
      \label{table_crossing_H}
   \end{table}

   Hence $\tau(H)=13$ and for every vertex cover $S$ of $G'$, $\vert S\cap H\vert \geq 13$. From a $\tau$-set $S$ of $G$, you can build a $\tau$-set $S'$ of $G'$ by taking all corresponding vertices of $S$ in $G'$ and exactly $13$ vertices for each copy of $H$. So if $\tau(G)=k$, then $\tau(G')=k+13d$, where $d$ is the number of copy of $H$ (i.e.\ the number of crossings). This completes the overview of the proof of Garey et al.\

   \begin{claim}\label{anticore1}
      if $\tau$-anticore$(G)\neq \emptyset$, then $\tau$-anticore$(G')\neq \emptyset$.
   \end{claim}

   Let $v$ be a vertex in $\tau$-anticore$(G)$. From Property \ref{tau_anticore} we know that $\tau(G) = \tau(G_v+u)+1$. Since $G'_v+u$ can be constructed from $G_v+u$, we have $\tau(G'_v+u)=\tau(G_v+u)+13$ and therefore $v\in \tau$-anticore$(G')$. Hence if a vertex $v$ belongs to $\tau$-anticore$(G)$, then $v$ also belongs to $\tau$-anticore$(G')$. This proves Claim \ref{anticore1}.

   \begin{claim}\label{anticore2}
      if $\tau$-anticore$(G)=\emptyset$, then $\tau$-anticore$(G')=\emptyset$.
   \end{claim}

   Suppose that $\tau$-anticore$(G)=\emptyset$. Let $ab, cd$ be two edges of $G$ that are crossing. We construct the graph $G_H$, a copy of $G$ with the crossing point replaced by the graph $H$ as described above. W.l.o.g.\ we can assume that there exist two minimum vertex cover $S_1, S_2$ of $G$ such that $a,c\in S_1$ and $b,d\in S_2$. Let $S_2'$ be the set of black vertices represented in Figure \ref{planar_H}. Note that $\vert S_2'\vert =13$ and that $S_2'$ is a minimum vertex cover of $H$. Also $S_2'$ is covering the edges $av_1$ and $cv_2$, while $S_2$ is covering the edges $bv_1'$ and $dv_2'$. Hence $S_2\cup S_2'$ is a vertex cover of $G_H$ and since $\tau(G_H)=\tau(G)+13$, it is minimum. Now consider the $180$ degree rotational symmetry of Figure \ref{planar_H}. Let $S_1'$ be the set of black vertices of this symmetry. Note that $S_1'$ is covering the edges $bv_1'$ and $dv_2'$ while $S_1$ is covering the edges $av_1$ and $cv_2$. Hence $S_1\cup S_1'$ is a vertex cover of $G_H$ and since $\tau(G_H)=\tau(G)+13$, it is minimum. From these two $\tau$-sets of $G_H$, we have $\tau$-anticore$(G_H)=\emptyset$. By iterating until there is no more edge crossing, we obtain $G'$ such that $\tau$-anticore$(G')=\emptyset$. This proves Claim \ref{anticore2}. \\

   Then from Claim \ref{anticore1} and \ref{anticore2}, it follows that $\tau$-anticore$(G') \neq \emptyset$ if and only if $\tau$-anticore$(G) \neq \emptyset$. This completes the proof.
\end{proof}

\begin{figure}[H]
   \centering
   \includegraphics[width=15cm, height=5cm, keepaspectratio=true]{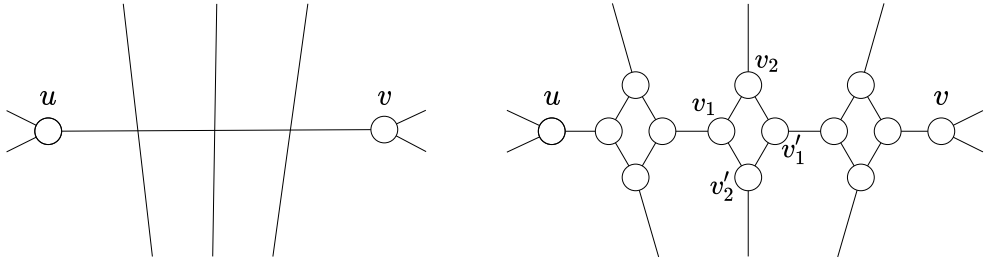}
   \caption{Removal of crossing edges of $G$ to construct $G'$ where each $C_4$ represent a copy of the graph $H$ of Figure \ref{planar_H}.}
   \label{crossing_edges}
\end{figure}

\begin{figure}[H]
   \centering
   \includegraphics[width=15cm, height=8cm, keepaspectratio=true]{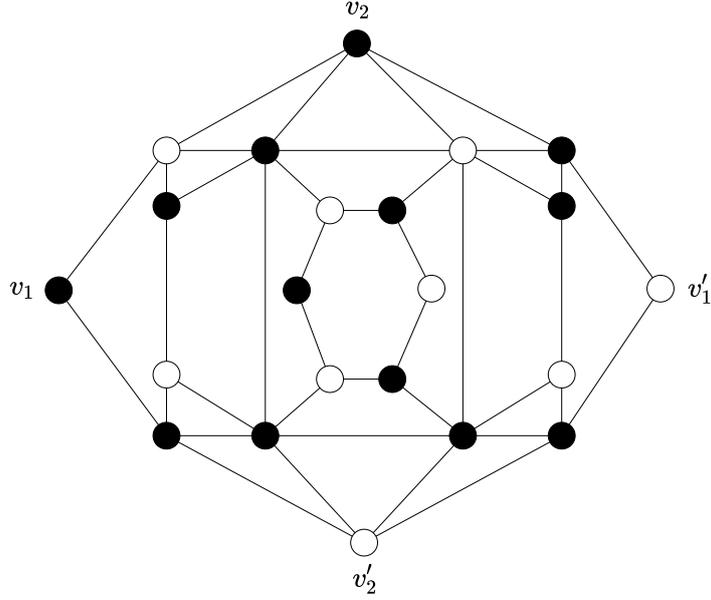}
   \caption{The planar graph $H$ with a minimum vertex cover represented by the black vertices.}
   \label{planar_H}
\end{figure}

Since for any graph $G$, we have $\tau$-anticore$(G)=\alpha$-core$(G)$, it follows:
\begin{coro}
   $\alpha$-\textsc{Core} is $\mathsf{NP}$-hard for planar graphs.
\end{coro}

\section{Complexity results}

In this section, we show that $1$-\textsc{Bondage} is $\mathsf{NP}$-hard in planar graphs with maximum degree $3$. The two following remarks will be of used in the proofs below.
\begin{rmk}\label{noedgecritical}
   Let $uv$ be an edge of a graph $G=(V,E)$. If there exists a minimum dominating set $S$ of $G$ such that $u,v\not\in S$ or $u,v\in S$, then $uv$ is not a $\gamma$-critical edge.
\end{rmk}

\begin{rmk}\label{novertexcritical}
   Let $v$ be a vertex of a graph $G=(V,E)$. If there exists a minimum dominating set $S$ of $G$ such that $v\not\in S$ and $\vert S\cap N(v)\vert\geq 2$, then every edge incident with $v$ is not $\gamma$-critical.
\end{rmk}

Using Theorem \ref{NPH_tau_anticore_planar}, we show the $\mathsf{NP}$-hardness of the $1$-\textsc{Bondage} problem in planar graphs.

\begin{theorem}\label{NPH_1-Bondage_planar}
   The $1$-\textsc{Bondage} problem is $\mathsf{NP}$-hard when restricted to planar graphs of maximum degree $3$.
\end{theorem}
\begin{proof}
   We give a polynomial reduction from $\tau$-\textsc{Anticore} which has been shown to be $\mathsf{NP}$-hard for planar graphs in Theorem \ref{NPH_tau_anticore_planar}. From any connected planar graph $G=(V,E)$, we build a connected planar graph $G'=(V',E')$ with maximum degree $3$. Let $n=\vert V(G)\vert$ and $m=\vert E(G)\vert$. \\

   For each vertex $v$ of $G$ we associate the connected component $G_v$. The component $G_v$ is as follows:

   \begin{itemize}
      \item An induced cycle $C_v$ of length $3l+3$, where $l=d_G(v)$, such that,\\ $C_v=v_0-\bar v_0-v_0'-v_1-\bar v_1-v_1'-\cdots-v_l-\bar v_l-v_l'-v_0$;
      \item An induced path $P_v=a-b-c-d$ that is connected to $C_v$ by the edge $av_0$;
      \item To each pair $v_i, v_i'$ of $C_v$, where $i=1,\ldots,d_G(v)-1$, is associated an induced paw $P^i=\{a_i,b_i,c_i,d_i\}$ where $\{b_i,c_i,d_i\}$ is a triangle, $a_ib_i$ is an edge, and $P^i$ is connected to $C_v$ by the edges $a_iv_i$ and $a_iv_i'$.
   \end{itemize}

   For each edge $uv$ of $G$, we create a copy of the graph $H_{uv}$ depicted in Figure \ref{H_uv}. Note that each component $G_v$ and $H_{uv}$ is planar with degree maximum $3$. Moreover each $\bar v_i$ of $G_v$, where $i=1,\ldots,d_G(v)$, has degree $2$. We show how the components connect to each other so that the graph is planar. Consider a planar representation of $G$. Each component $G_v$ is assigned the same position on the plane as its associated vertex $v$ of $G$. We do similarly with each component $H_{uv}$ and its associated edge $uv$ of $G$ so that no components overlap. Embed each induced cycle $C_v$ in the plane as a circle so that the path $P_v$ and the paws $P^i$, where $i=1,\ldots,d_G(v)-1$, are inside. Then for each vertex $v$ of $G$, iterate counterclockwise on the edges incident to $v$ from $i=1$ to $i=d_G(v)$. At each step, let $uv$ be the considered edge incident to $v$, add the edge $\bar v_ih_v$ where $h_v\in V(H_{uv})$. Hence $G'$ is planar with $\Delta(G')\leq 3$. This completes the construction. \\

   In the first part we will prove that $\gamma(G')=6m+n+\tau(G)$. In the last part we will prove that $b(G')=1$ if and only if $\tau$-anticore$(G) \neq \emptyset$. \\

   We prove the following:

   \begin{claim}\label{planar1}
      Let $S'$ be a dominating set of $G'$. For each $G_v$, $\vert S'\cap V(G_v)\vert \geq 2d_G(v)+1$.
   \end{claim}

   Let $G_v$ and $l=d_G(v)$. Let $T$ be a dominating set of $G'\setminus V(G_v)$. From $T$ we construct $S'$ a dominating set of $G'$. Since we intend to prove a lower bound on $\vert S'\cap V(G_v)\vert$, we can assume that $N(V(G_v))\subset T$. Let $S'=T$. Note that the set of vertices $\{\bar v_i \mid i=1,\ldots,l\}$ of $G_v$ is dominated by $S'$.
   Since $v_0-a-b-c-d$ is an induced path in $G'$, such that $N(V(P_v))=v_0$, we add $\{c,v_0\}$ to $S'$.
   For each paw $P^i=\{a_i,b_i,c_i,d_i\}$, where $i=1,\ldots,l-1$, since $G'[\{b_i,c_i,d_i\}]$ is a triangle and that $N(\{b_i,c_i,d_i\})=\{a_i\}$, where $a_ib_i\in E$, we add $b_i$ to $S'$. At this step $\vert S'\cap V(G_v)\vert = l+1$. It remains to dominate $I=\{v_{i-1}',v_{i} \mid i=1,\ldots,l\}$. Note that $I$ is the disjoint unions of $l$ induced paths of size $2$. If neither $v_{i-1}'$, nor $v_{i}$ is dominating, then we can assume that $\bar v_{i-1},\bar v_{i}$ must be dominating (note that $\bar a_{i-1}$ and $\bar a_{i}$ is an equivalent pair to dominate $v_{i-1}',v_{i}$). Hence we must add at least $l$ vertices to $S'$. From our previous arguments, if we add exactly $l$ vertices to $S'$, then either $v_{i-1}'\in S'$, or $v_i\in S'$; and $\bar v_{0}, \bar v_i\not\in S'$, where $i=1,\ldots,l$. This proves Claim \ref{planar1}. \\

   Let $S_v$ be a dominating set of a component $G_v$. We give the following definitions:
   \begin{itemize}
      \item If $\vert S_v\vert=2d_G(v)+1$, then $S_v$ is called a \textit{non-dominating configuration} of $G_v$. Note that from our previous arguments, if $S_v$ is a non-dominating configuration, then $S_v\cap \{\bar v_i \mid i=0,\ldots,d_G(v)\} =\emptyset$. An example of such a configuration is $S_v=\{c,v_0,v_l\}\cup \{b_i,v_i \mid i=1,\ldots,l-1\}$, where $l=d_G(v)$.
      \item If $\vert S_v\vert\geq 2d_G(v)+2$ and $\bar v_i\in S_v$, for some $\bar v_i\in V(G_v)$, then $S_v$ is called a \textit{semi-dominating configuration} of $G_v$.
      \item If $\vert S_v\vert=2d_G(v)+2$ and $\{\bar v_i \mid i=1,\ldots,d_G(v)\}\subset S_v$, then $S_v$ is called a \textit{dominating configuration} of $G_v$. Note that a dominating configuration is a restricted case of a semi-dominating configuration. An example of a dominating configuration is $S_v=\{c,v_0,v_0',\bar v_l\}\cup \{b_i, \bar v_i \mid i=1,\ldots,l-1\}$, where $l=d_G(v)$.
   \end{itemize}

   From Claim \ref{planar1}, given a dominating set $S'$ of $G'$, we have for every component $G_v$ that either $S_v=S'\cap V(G_v)$ is a non-dominating configuration or a semi-dominating configuration. \\

   Let $H_{uv}$ be the associated component of an edge $uv$ of $G$, and $\bar u\in V(G_u)$, $\bar v\in V(G_v)$, such that $h_u\bar u, h_v\bar v \in E$. Let $S'$ be a minimum dominating set of $G'$. One can easily check the following claim.

   \begin{claim}\label{planar2}
      If $\bar u,\bar v\not\in S'$, then $\vert S'\cap V(H_{uv})\vert \geq 3$. Else $\vert S'\cap V(H_{uv})\vert= 2$.
   \end{claim}

   Let $S'$ be a dominating set of $G'$. Consider the sets $S'_v=S'\cap V(G_v)$, $S'_u=S'\cap V(G_u)$ and $S'_{uv}=S'\cap V(H_{uv})$. By Claim \ref{planar2}, $\vert S'_{uv}\vert = 2$ if and only if $S'_u$ or $S'_v$ is a semi-dominating configuration, where $\bar u_i\in S'$ or $\bar v_i\in S'$. In this case, we say that $S_{uv}'$ is a \textit{covered configuration} and that $H_{uv}$ is covered by $S'$. When $\bar u_i\in S'$ (resp. $\bar v_i\in S'$), then we say that $G_u$ (resp. $G_v$) is \textit{covering} $H_{uv}$. Examples of covered configurations where $G_u$ or $G_v$ cover $H_{uv}$ are displayed in Figure \ref{H_uv_dom}, \ref{H_uv_dom3}, and \ref{H_uv_dom2}.

   We define as $S'_G$ the set of $G_v$ such that $S_v'$ is a semi-dominating configuration in $S'$, and $S'_H$ the set of $H_{uv}$ that are not covered by $S'$. From Claim \ref{planar2}, if a component $H_{uv}$ is not covered, then $\vert S'_{uv}\vert \geq 3$. The following lower bound follows:
   \begin{equation}\label{gamma_lbound}
      \begin{split}
      \vert S\vert & \geq 2m + \sum_{v\in V(G)} (2d_G(v)+1) + \vert S_G\vert + \vert S_H\vert \\
         & \geq 6m + n + \vert S_G\vert + \vert S_H\vert
      \end{split}
   \end{equation}

    We say that a dominating set $S'$ of $G'$ is the \textit{associated dominating set} of a minimum vertex cover $S$ of $G$, if the following is true:
    \begin{itemize}
       \item If $v\not\in S$, then $S'_v$ is a non-dominating configuration of $G_v$;
       \item If $v\in S$, then $S'_v$ is a dominating configuration of $G_v$;
       \item For each edge $uv$, then $S'_{uv}$ is a covered configuration of $H_{uv}$ (depending on which of $G_u$ or $G_v$ is covering $H_{uv}$).
    \end{itemize}

   We prove the following:

   \begin{claim}\label{planar3}
      $\gamma(G') \leq 6m + n + \tau(G)$.
   \end{claim}

   Let $S'$ be the associated dominating set of a vertex $S$ of $G$. Since each $H_{uv}$ is covered by $G_u$ or $G_v$, we have $S'$ a dominating set of $G'$ such that $\vert S'_G\vert =\tau(G)$ and $\vert S'_H\vert =0$. It follows that $\vert S' \vert = 6m+n+\tau(G)$. This proves Claim \ref{planar3}.



   \begin{claim}\label{planar4}
      $\gamma(G') \geq 6m + n + \tau(G)$.
   \end{claim}

   Let $p=6m+n+\tau(G)$. By contradiction, suppose there exists a dominating set $S'$ of $G'$ such that $\vert S'\vert < p$. From (\ref{gamma_lbound}) we have $\vert S'_G\vert + \vert S'_H\vert  < \tau(G)$. From $S'$ we construct a vertex cover $S$ of $G$. For each component $G_v$ of $S'_G$, we add $v$ to $S$. Now it remains to cover the edges associated with the components of $S'_H$. For each component $H_{uv}$ of $S'_H$, let $uv$ be its associated edge in $G$: we add one of its endpoints, say $u$, to $S$. Thus $S$ is a vertex cover of $G$ such that $\vert S\vert < \tau(G)$, a contradiction. This proves Claim \ref{planar4}. \\

   From Claim \ref{planar3} and \ref{planar4}, it follows:

   \begin{claim}\label{planar5}
      $\gamma(G') = 6m + n + \tau(G)$.
   \end{claim}

   From the arguments used to prove Claim \ref{planar3}, and from Claim \ref{planar5}, we have that an associated dominating set of a minimum vertex cover is a minimum dominating set. \\

   In the remaining part, we show that $b(G')=1$ if and only if $\tau$-anticore$(G)\neq \emptyset$. First, we prove that if $\tau$-anticore$(G)=\emptyset$, then $b(G')>1$. Therefore we show that no edge with an endpoint in $H_{uv}$ and no edge in $G_v$ is $\gamma$-critical in $G'$.

   \begin{claim}\label{planar6}
      If $\tau$-anticore$(G)=\emptyset$, then for every $H_{uv}$, no edge with an endpoint in $H_{uv}$ is $\gamma$-critical in $G'$.
   \end{claim}

   Let $uv\in E$ and $H_{uv}$ be its associated component. Let $S$ be a minimum vertex cover of $G$ and $S'$ be its associated dominating set in $G'$.

   First suppose that $u,v\in S$. Consider the covered configuration $S'_{uv}$ highlighted in Figure \ref{H_uv_dom3}. From Remark \ref{noedgecritical} and \ref{novertexcritical}, no bold edge is $\gamma$-critical. Then for the three remaining non-bold edges, we take the symmetric covered configuration highlighted in Figure \ref{H_uv_dom4}. Hence from Remark \ref{noedgecritical} and \ref{novertexcritical}, no edge with an endpoint in $H_{uv}$ is $\gamma$-critical in $G'$.

   Second, we may assume that $u\in S$ and $v\not\in S$. Since $\tau$-anticore$(G')=\emptyset$, there is another minimum vertex cover $T$ of $G$ such that $u\not\in T$ and $v\in T$. Let $T'$ be the associated dominating set of $T$ in $G'$. Then $G_u$ (resp. $G_v$) is covering $H_{uv}$ in $S'$ (resp. $T'$). Therefore we consider the highlighted vertices of Figure \ref{H_uv_dom} (resp. Figure \ref{H_uv_dom5}) for the covered configuration $S_{uv}'$ (resp. $T_{uv}'$). From Remark \ref{noedgecritical} and \ref{novertexcritical}, no bold edge is $\gamma$-critical. Hence no edge with an endpoint in $H_{uv}$ is $\gamma$-critical. This proves Claim \ref{planar6}.

   \begin{claim}\label{planar7}
      If $\tau$-anticore$(G)=\emptyset$, then for every component $G_{v}$, no edge in $G_v$ is $\gamma$-critical in $G'$.
   \end{claim}

   Let $v\in V$, $l=d_G(v)$, and $G_{v}$ be its associated component.

   Suppose that $v\not\in \tau$-core$(G)$. Since $\tau$-anticore$(G)=\emptyset$, there exists $S$ and $T$, two $\tau$-sets of $G$, such that $v\not\in S$ and $v\in T$. Let $S', T'$ be two associated dominating sets of $S,T$, respectively. Note that each $H_{uv}$ is covered by $G_u,G_v$, in $S',T'$, respectively. Hence we can assume that $S_{uv}'$ is as depicted in Figure \ref{H_uv_dom2}. Note that $S_v'$ corresponds to a non-dominating configuration while $T_v'$ corresponds to a dominating configuration of $G_v$. We give a non-dominating configuration for $S_v'$ and a dominating configuration for $T_v'$:

   \begin{enumerate}
      \item $S_v'=Z\cup \{c,v_0\}\cup \{w_i \mid w_i\in \{b_i,c_i,d_i\}, i=1,\ldots,l-1\}$ where $Z=\{v_i \mid i=1,\ldots,l\}$ or $Z=\{v_i' \mid i=0,\ldots,l-1\}$;
      \item $T_v'=W\cup \{\bar v_l\}\cup \{b_i,\bar v_i \mid i=1,\ldots,l-1\}$ where $W=\{c,v_0,v_0'\}$ or $W=\{d,b,\bar v_0\}$ or $W=\{c,v_0,\bar v_0\}$.
   \end{enumerate}

   Hence we can take any previous configurations of $S_v',T_v'$ in $S',T'$, respectively, such that $S'$ and $T'$ are still $\gamma$-sets of $G'$. Then from Remark \ref{noedgecritical} and \ref{novertexcritical}, one can check that no edge of $G_v$ is $\gamma$-critical in $G'$.

   Now we can suppose that $v\in \tau$-core$(G)$. Since $\tau$-anticore$(G)=\emptyset$, it follows that for each edge $uv$ of $G$, there exists $S$ a minimum vertex cover of $G$ such that $u,v\in S$. Let $u$ be a neighbor of $v$, and let $S$ be a minimum vertex cover of $G$ such that $u,v\in S$. Let $S'$ be an associated dominating set in $G'$. Note that $S_u,S_v$ correspond to a dominating configuration of $G_u,G_v$, respectively, and therefore $H_{uv}$ is covered by both $G_u$ and $G_v$. Hence we replace the vertices of $S_{uv}'$ by the dominating black vertices of $H_{uv}$ depicted in Figure \ref{H_uv_dom2}. Thus the vertex $h_v$ dominates the vertex $\bar v_j$ of $G_v$, where $j\in \{1,\ldots,l\}$. We exposed some dominating sets of $G_v$ to replace $S_v'$ in $S$. Note that item 3.\ represents some dominating configurations of $G_v$, while item 4.\ and 5.\ are semi-dominating configurations of size $2d_G(v)+2$ that cover every component $H_{u'v}$, where $u'\neq u$.

   \begin{enumerate}
      \setcounter{enumi}{2}
      \item $S_v'=W\cup \{\bar v_l\}\cup \{b_i,\bar v_i \mid i=1,\ldots,l-1\}$ where $W=\{c,v_0,v_0'\}$ or $W=\{d,b,\bar v_0\}$ or $W=\{c,v_0,\bar v_0\}$;
      \item $S_v'=\{a_j,w_j\}\cup \{c,v_0,v_0',\bar v_l\}\cup \{b_i,\bar v_i \mid i=1,\ldots,l-1, i\neq j\}$, if $j\neq l$, where $w_j\in \{b_j,c_j,d_j\}$;
      \item $S_v'=\{v_j\}\cup \{c,v_0,v_0'\}\cup \{b_i,\bar v_i \mid i=1,\ldots,l-1, i\neq j\}$, if $j=l$.
   \end{enumerate}

   Hence we can take any exposed configurations of $S_v'$ in $S'$ such that $S'$ is still a $\gamma$-set of $G'$. Then from Remark \ref{noedgecritical} and \ref{novertexcritical}, one can check the following: from item 3.\ no edge incident to a vertex $\{a,b,c,d,v_0,\bar v_0,v_0'\}$ is $\gamma$-critical in $G'$; from item 4.\ no edge incident to $\{a_j,b_j,c_j,d_j,v_j,\bar v_j,v_j'\}$ is $\gamma$-critical in $G'$; from 5.\ no edge incident to $\{v_j,\bar v_j,v_j'\}$ is $\gamma$-critical in $G'$. Since we can select any neighbor of $v$ to construct $S'$, these configurations can be applied to any $j\in \{1,\ldots,l\}$, and therefore $G_v$ has no $\gamma$-critical edge in $G'$. This proves Claim \ref{planar7}. \\

   Since every edge of $G'$ is either in a component $G_v$ or has an endpoint in a component $H_{uv}$, it follows from Claim \ref{planar6} and \ref{planar7}, that if $\tau$-anticore$(G)=\emptyset$, then $G'$ has no $\gamma$-critical edge. Hence it follows:

   \begin{claim}\label{planar8}
      If $\tau$-anticore$(G)=\emptyset$, then $b(G')>1$.
   \end{claim}

   It remains to prove the following:

   \begin{claim}\label{planar9}
      If $\tau$-anticore$(G)\neq \emptyset$, then $b(G')=1$.
   \end{claim}

   Let $v\in \tau$-anticore$(G)$. We show that there is no $\gamma$-set of $G'$ such that $G_v$ has a semi-dominating configuration. By contradiction, let $S'$ be a $\gamma$-set of $G'$ such that $S_v'$ is a semi-dominating configuration of $G_v$. From $S'$ we construct a minimum vertex cover $S$ of $G$. For each component $G_u$ of $S_G'$, we add the vertex $u\in V(G)$ to $S$. For each component $H_{uv}$ of $S_H'$, we take one of its endpoints, say $u\in V(G)$, and add it to $S$. From \ref{gamma_lbound} and Claim \ref{planar5} we have $\vert S_G \vert + \vert S_H\vert = \tau(G)$. Thus $\vert S\vert =\tau(G)$. Yet $S$ is a minimum vertex cover of $G$ where $v\in S$, a contradiction. Hence there is no $\gamma$-set of $G'$ where $G_v$ has a semi-dominating configuration.

   Now we are ready to highlight a $\gamma$-critical edge in $G'$. Let $S'$ be a $\gamma$-set of $G'$. From the previous arguments used to prove Claim \ref{planar1}, we have only two possible configurations of $G_v$ in $S_v'$, that are:

   \begin{enumerate}
      \item $S_v'=\{c,v_0,v_l\}\cup \{v_i, w_i \mid w_i\in \{b_i,c_i,d_i\}, i=1,\ldots,d_G(v)-1\}$;
      \item $S_v'=\{c,v_0,v_0'\}\cup \{v_i', w_i \mid w_i\in \{b_i,c_i,d_i\}, i=1,\ldots,d_G(v)-1\}$.
   \end{enumerate}

   Then in every $\gamma$-set $S'$ of $G'$, we have $c\in S'$ and $d\not\in S'$, where $c,d\in V(G_v)$ (precisely $c,d\in V(P_v)$). Recall that $d$ is a leaf. Hence $d$ is dominated exclusively by $c$ in all $\gamma$-set of $G'$, and therefore the edge $cd$ is $\gamma$-critical in $G'$. This proves Claim \ref{planar9}. \\

   From Claim \ref{planar8} and \ref{planar9}, it follows that $b(G')=1$ if and only if $\tau$-anticore$(G)\neq \emptyset$. This completes the proof.
\end{proof}

\begin{figure}[H]
   \centering
   \includegraphics[width=15cm, height=4.2cm, keepaspectratio=true]{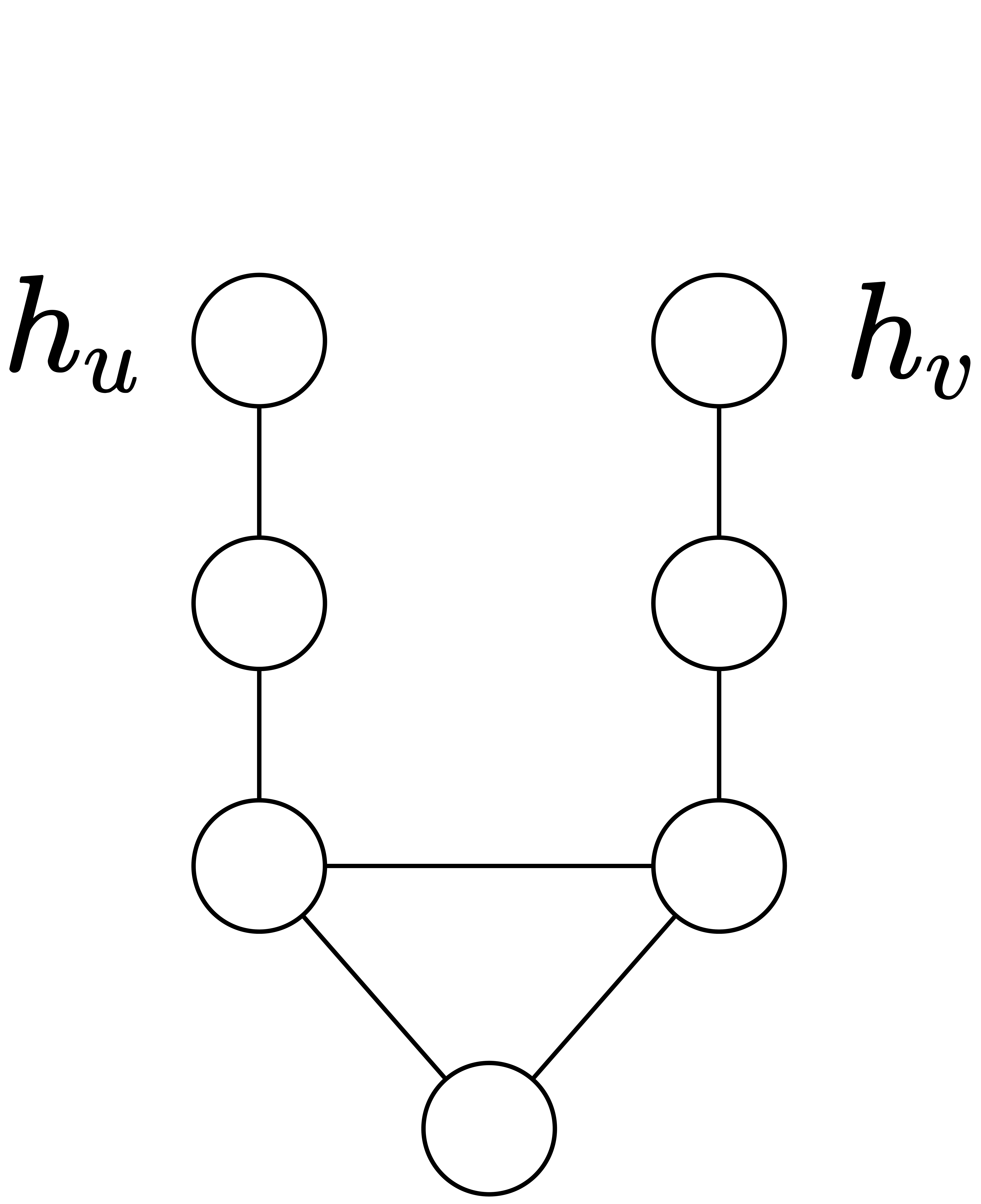}
   \caption{The component $H_{uv}$ associated with an edge $uv$ of $G$.}
   \label{H_uv}
\end{figure}

\begin{figure}[H]
   \centering
   \begin{subfigure}{.25\textwidth}
      \centering
      \includegraphics[width=15cm, height=4.2cm, keepaspectratio=true]{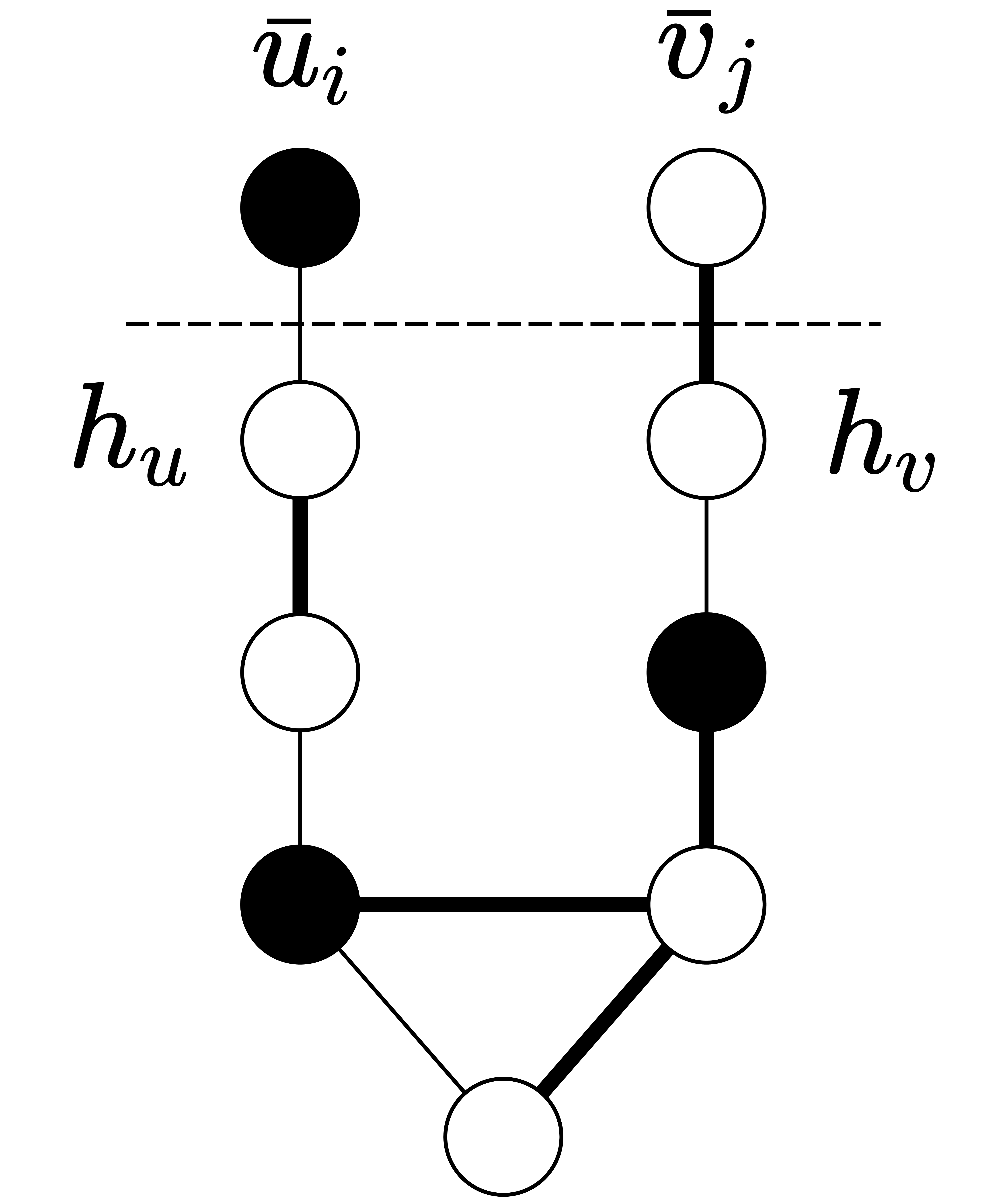}
      \caption{}
      \label{H_uv_dom}
   \end{subfigure}\hfill%
   \begin{subfigure}{.25\textwidth}
      \centering
      \includegraphics[width=15cm, height=4.2cm, keepaspectratio=true]{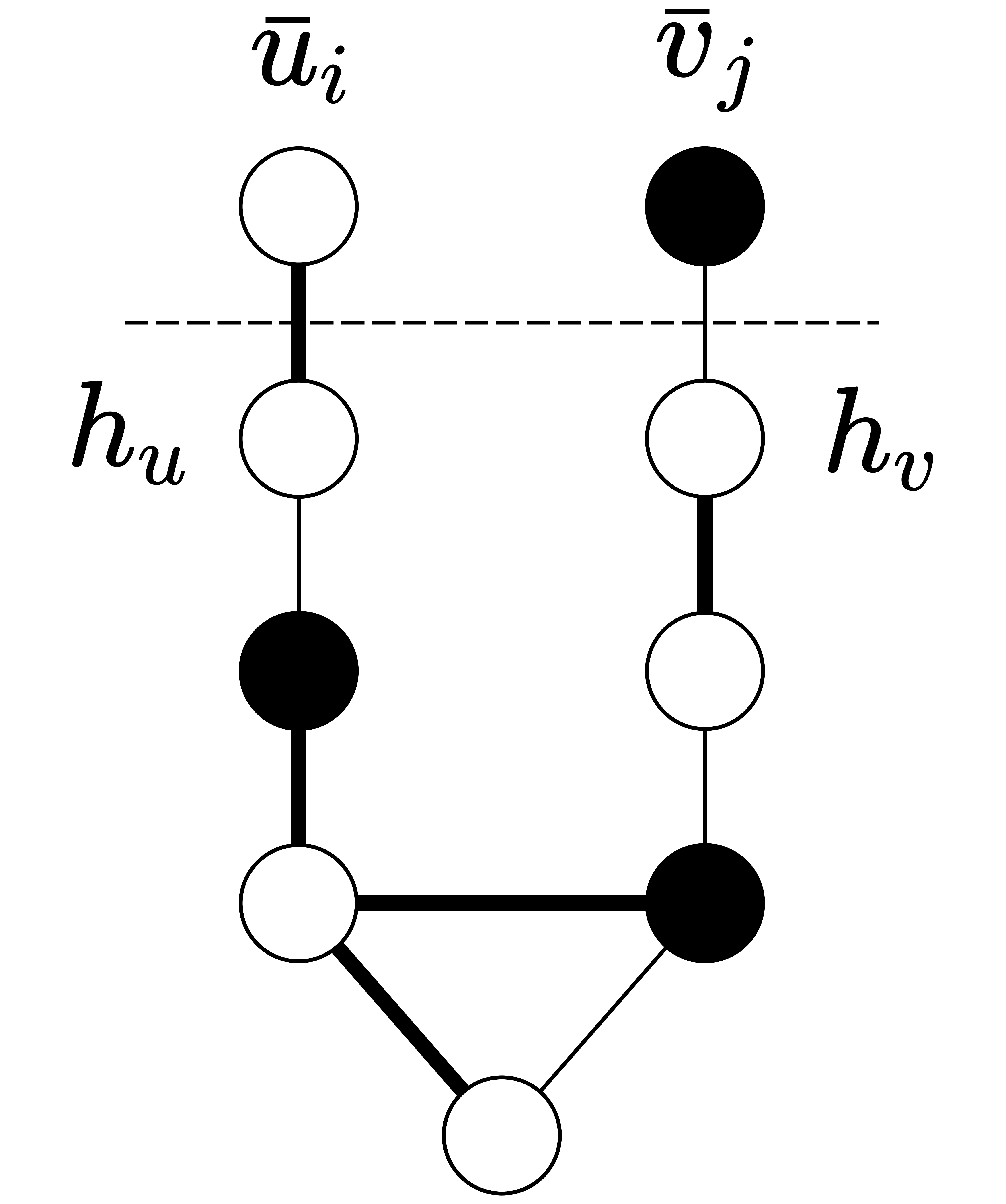}
      \caption{}
      \label{H_uv_dom5}
   \end{subfigure}\hfill%
   \begin{subfigure}{.25\textwidth}
      \centering
      \includegraphics[width=15cm, height=4.2cm, keepaspectratio=true]{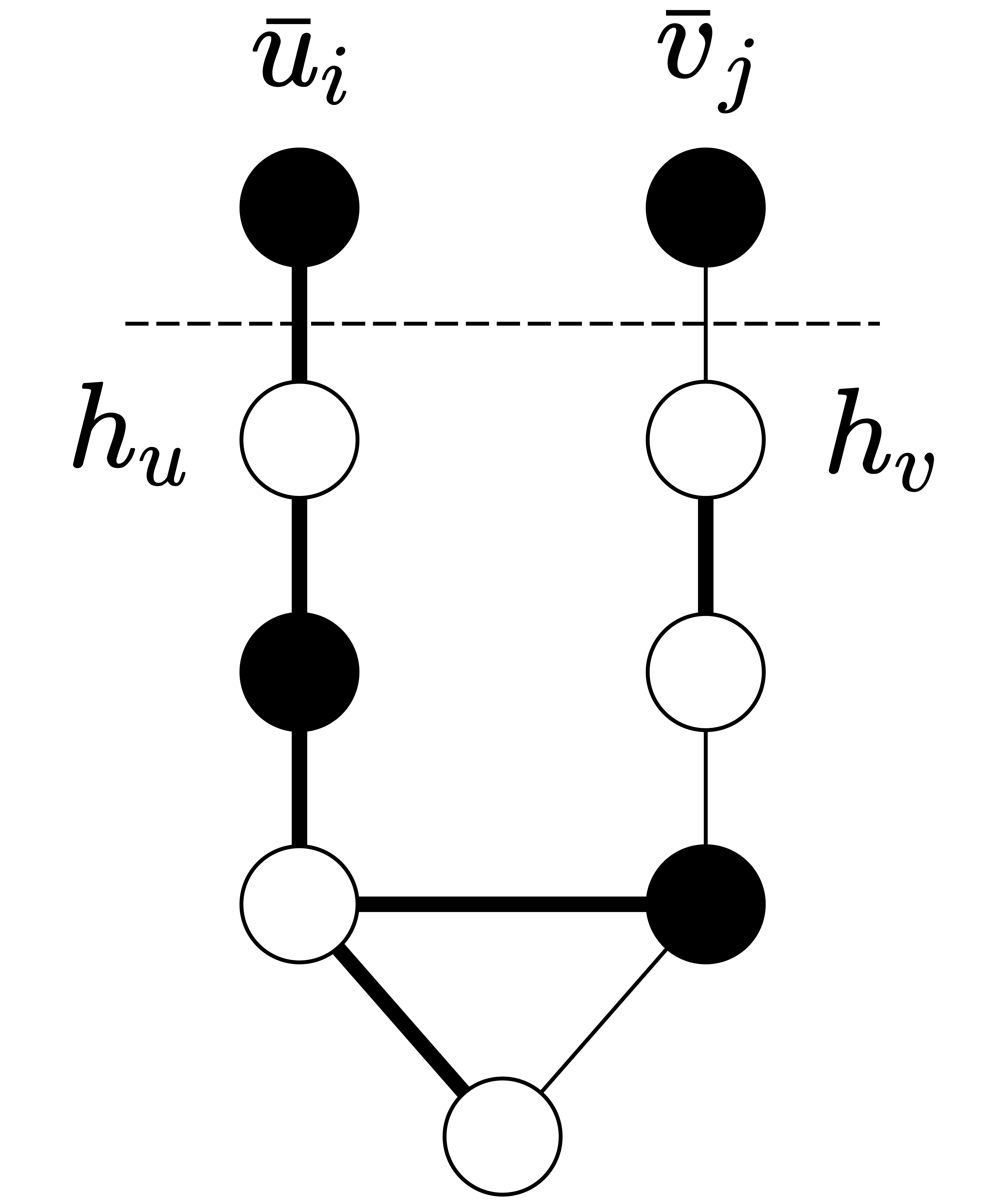}
      \caption{}
      \label{H_uv_dom3}
   \end{subfigure}\hfill%
   \begin{subfigure}{.25\textwidth}
      \centering
      \includegraphics[width=15cm, height=4.2cm, keepaspectratio=true]{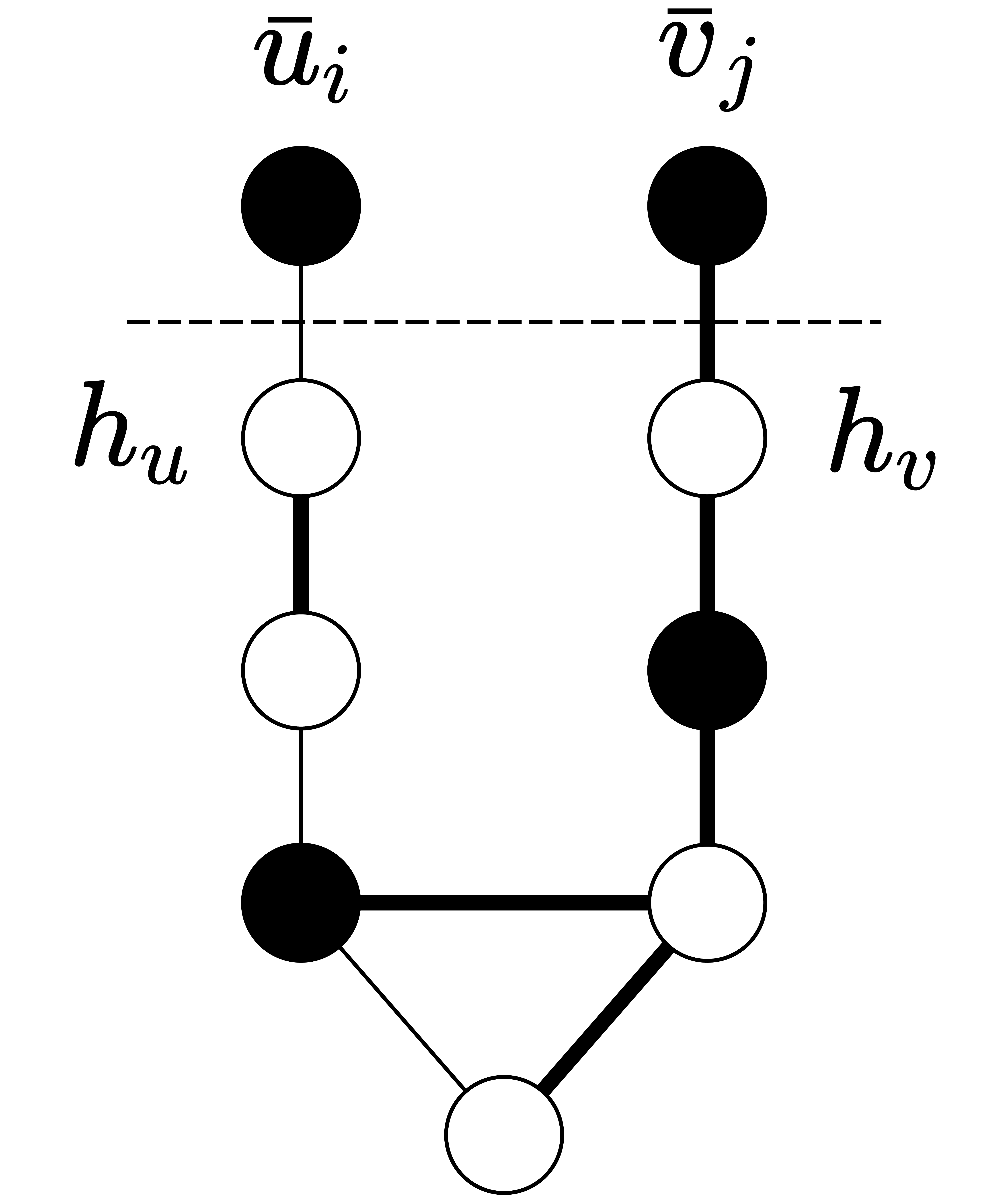}
      \caption{}
      \label{H_uv_dom4}
   \end{subfigure}\hfill%
   \caption{(a) A covered configuration of $H_{uv}$ when $G_u$ is covering; (b) a covered configuration of $H_{uv}$ when $G_v$ is covering; (c,d) two covered configuration of $H_{uv}$ when $G_u$ and $G_v$ are covering.}
\end{figure}

\begin{figure}[H]
   \centering
   \includegraphics[width=15cm, height=4.2cm, keepaspectratio=true]{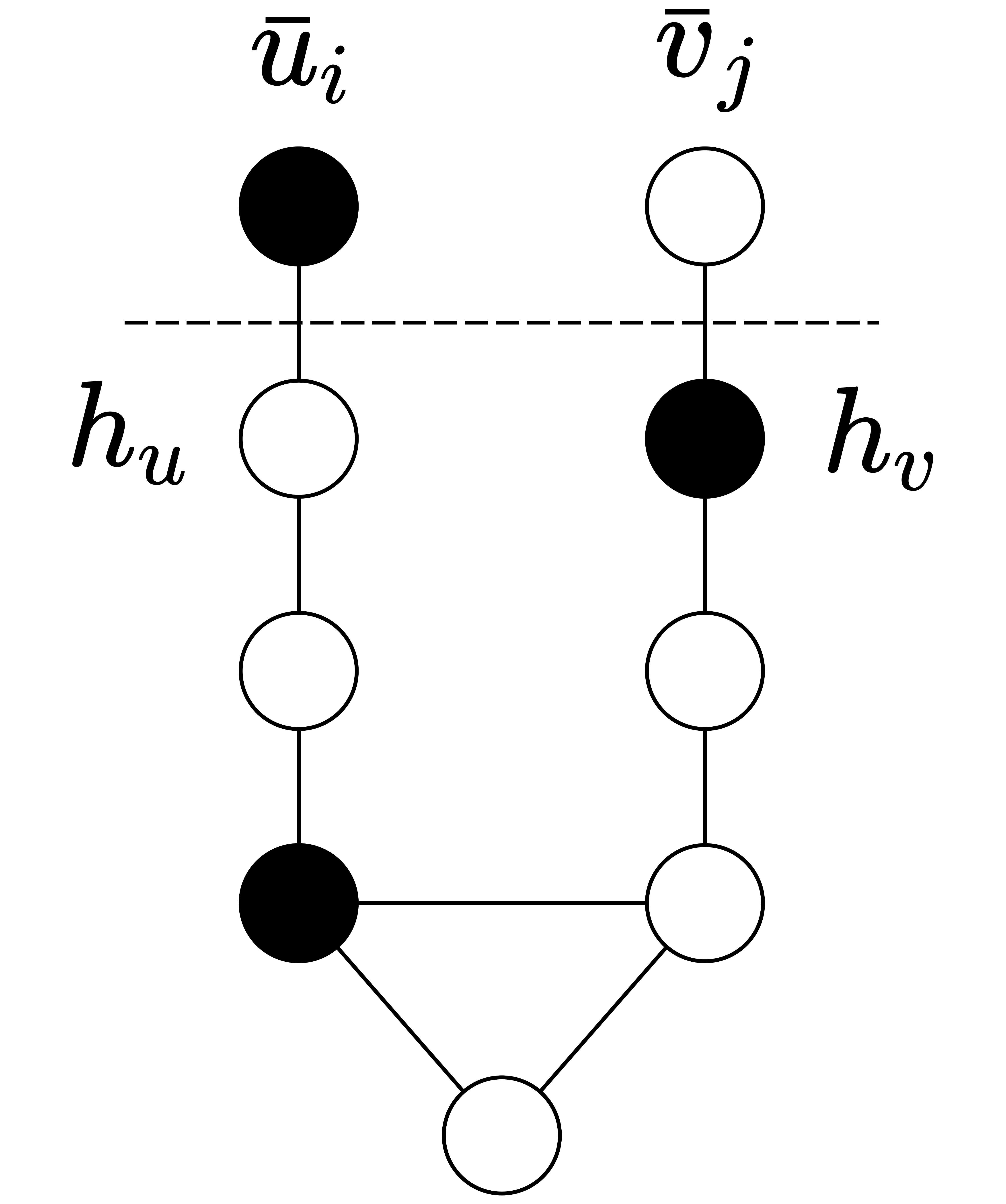}
   \caption{a covered configuration of $H_{uv}$ when $G_u$ is covering and such that $h_v$ is dominating $\bar v_j$.}
   \label{H_uv_dom2}
\end{figure}



Extending the arguments of the proof of Theorem \ref{NPH_1-Bondage_planar}, we can obtain the following:
\begin{theorem}\label{NPH_1-Bondage_planar_claw}
   $1$-\textsc{Bondage} is $\mathsf{NP}$-hard when restricted to planar claw-free graphs of maximum degree $3$.
\end{theorem}
\begin{proof}
   From the graph $G'$ of the proof of Theorem \ref{NPH_1-Bondage_planar}, we construct a planar claw-free graph $H'$ with degree maximum $3$. \\

   First we define the operation $O_c$ to construct an intermediary planar graph $H$ from $G$, such that $\Delta(H)\leq 3$, and with $k-1$ induced claws, where $k$ is the number of induced claws in $G'$. Second we show that $\gamma(H)=\gamma(G')+2$. Last we show that $b(H)=b(G')$. \\

   The operation $O_c$ is as follows: let $v$ be a vertex at the center of an induced claw in $G'$. We replace $v$ by a component $H_v$ as depicted in the right of Figure \ref{claw}. Let $H$ be the graph obtained from this operation. Note that $H$ is planar with $\Delta(H)\leq 3$, and that $H$ has $k-1$ claws.

   The component $H_v$ satisfies the following properties. We have $\gamma(H_v)=3$ and $\{a,b,c\}$ is a $\gamma$-set. The graphs $H_{v}-a$, $H_{v}-b$, $H_{v}-c$ have a unique $\gamma$-set of size two, that are, $\{v_1,v_4\}, \{v_2,v_5\}, \{v_3,v_6\}$, respectively. With one vertex it is not possible to dominate $H_v-\{a,b,c\}$. \\

   First we prove the following:

   \begin{claim}\label{claw1}
      $\gamma(H)\leq \gamma(G')+2$
   \end{claim}

   Let $S'$ be a minimum dominating set of $G'$. From $S'$ we construct a dominating set $T$ of $H$. Let $T=S'$. Suppose that $v\not\in S'$. If $u_1\in S'$, then we add $v_3,v_6$ to $T$. Hence $T$ dominates $H$. The case $u_2\in S'$ and $u_3\in S'$ are symmetric. Now $v\in S'$. Then remove $v$ from $T$ and add the vertices $a,b,c$. Hence $T$ is a dominating set of $H$, and $T$ dominates $u_1,u_2,u_3$. So $\gamma(H)\leq \gamma(G')+2$. This proves Claim \ref{claw1}

   \begin{claim}\label{claw2}
      $\gamma(H) = \gamma(G')+2$
   \end{claim}

   Let $T$ be a minimum dominating set of $H$. From $T$ we construct a dominating set $S'$ of $G'$. Suppose that $\vert T\vert < \gamma(G')+2$. From our previous arguments $\vert T\cap V(H_v)\vert\geq 2$. If $\vert T\cap V(H_v)\vert = 2$, then $a,b,c\not\in T$ and $T\cap \{u_1,u_2,u_3\}\neq \emptyset$. Hence $T\cap V(G')$ is a $\gamma$-set of $G'$, a contradiction. Else $\vert T\cap V(H_v)\vert\geq 3$, and then $T\cap V(G')\cup \{v\}$ is a $\gamma$-set of $G'$ another contradiction. Therefore $\gamma(H)\geq \gamma(G)+2$. So from Claim \ref{claw1}, this proves Claim \ref{claw2}. \\

   Let $S_G$ and $S'_G$ be two minimum dominating sets of $G'$. We show that no edge incident with a vertex of $H_v$ in $H$ is $\gamma$-critical, when $S_G$ and $S'_G$ correspond to one of the two cases described below.

   \begin{enumerate}
      \item[(a)] $u_i,v\in S_G$, and $u_i\in S_G'$ such that $v\not\in S_G'$, where $i\in \{1,2,3\}$.

      Suppose that $u_1=u_i$. From $S_G, S_G'$ we build $S_H, S'_H$, respectively, as follows. Let $S_H=S_G$ and $S'_H=S'_G$. We remove $v$ from $S_H$ and add $\{b,c,w\}$, where $w\in \{a,v_1,v_2\}$. We add $v_3,v_6$ to $S'_H$. From Claim \ref{claw2}, $S_H$ and $S'_H$ are two $\gamma$-sets of $H$. From Remark \ref{noedgecritical} and \ref{novertexcritical}, no edge with an endpoint in $H_{v}$ is $\gamma$-critical. The cases $u_2=u_i$ and $u_3=u_i$ are symmetric.

      \item[(b)] $u_i,u_j\in S_G$, and $u_i,u_l\in S'_G$, such that $v\not\in S_G\cup S'_G$, where $i,j,l\in \{1,2,3\}$, $i\neq j\neq l$.

      Suppose that $u_i=u_1$, $u_j=u_2$ and $u_l=u_3$. From $S_G, S'_G$ we build $S_H, S'_H$, respectively, as follows. Let $S_H=S_G$ and $S'_H=S'_G$. We add $\{v_1,v_4\}$ or $\{v_3,v_6\}$ to $S_H$. We add $\{v_2,v_5\}$ or $\{v_3,v_6\}$ to $S'_H$. From Claim \ref{claw2}, $S_H$ and $S'_H$ are two $\gamma$-sets of $H$. From Remark \ref{noedgecritical} and \ref{novertexcritical}, no edge with an endpoint in $H_{v}$ is $\gamma$-critical. The remaining cases are symmetric.
   \end{enumerate}

   We are now reading to prove the following:

   \begin{claim}\label{claw3}
      If $b(G')\geq 2$, then no edge with an endpoint in $H_v$ is $\gamma$-critical in $H$.
   \end{claim}

   When $b(G')\geq 2$ we show that for each vertex $v\in V(G')$ that is at the center of a claw, there is $S_G$ and $S'_G$ as described in one of two items described above. We do so by referencing the five enumerated dominating and non-dominating configurations of a component $G_v$ that are exposed in the proof of Theorem \ref{NPH_1-Bondage_planar}, see the proof of Claim \ref{planar9} items 1.\ to 5.\ page 11. These dominating configurations are used to construct $\gamma$-sets of $G'$ when $b(G')=2$ (i.e.\ $\tau$-anticore$(G)=\emptyset$). Since each center of a claw in $G'$ is in a component $G_v$, we consider the configurations for each claw of items 1.\ and 2.\ and check if there are some configurations as described in items (a) or (b). Then we do similarly for the configurations of items 3.\ 4.\ and 5. For each claw, we give the configurations that correspond to case (a) or (b).

   First, we focus on the claw $G'[\{a,v_0,\bar v_0,v_l'\}]$, where $v_0$ is at its center. From item 2.\ we have $\bar v_0, v_0\in S_G$ and $\bar v_0\in S_G'$, $v_0\not\in S_G'$. This correspond to case (a) where $u_i=\bar v_0$ and $v=v_0$. From item 3.\ we have $v_0,\bar v_0\in S_G$ and $\bar v_0\in S_G'$, $v_0\not\in S_G'$. This correspond to case (a) where $u_i=\bar v_0$ and $v=v_0$.

   Now, we deal with the claw $G'[\{a_i,b_i,v_i,v_i'\}]$, where $a_i$ is at its center. From item 1.\ we have $b_i,v_i\in S_G$ and $b_i,v_i'\in S_G'$ such that $a_i\not\in S_G\cup S_G'$. This correspond to case (b) where $u_i=b_i$, $u_j=v_i$, $u_l=v_i'$, and $v=a_i$. From item 4.\ we have $b_i,a_i\in S_G$ (see $b_j,a_j$) and $b_i\in S_G'$, $a_i\not\in S_G'$ (see $b_i, a_i$). Note that these configurations work for $i,j=1,\ldots,l$, $i\neq j$. This correspond to case (a), where $u_i=b_i$ and $v=a_i$.

   Last, it remains the claw $G'[\{h_v,v_i,v_i',\bar v_i\}]$, where $\bar v_i$ is at its center. From item 1.\ we have $h_vv_i\in S_G$ or $h_vv_i'\in S_G'$ and $\bar v_i\not\in S_G\cup S_G'$. We recall that $h_v$ is dominating because item 1.\ is a non-dominating configuration (see the paragraph above item 1.\ where it is mentioned that $S_{uv}'$ is as depicted in Figure \ref{H_uv_dom2}). Hence this correspond to case (b) where $u_i=h_v$, $u_j=v_i$ and $u_l=v_i'$. From item 3.\ $h_v,\bar v_i \in S_G$. From item 4.\ $h_v\in S_G$, $\bar v_i\not\in S_G'$ (see $\bar v_j$). Note that these configurations work for $i,j=1,\ldots,l$, $i\neq j$. For similar reasons, the vertex $h_v$ is dominating (see the paragraph above item 3.). Thus this correspond to case (a) where $u_i=h_v$ and $v=v_0$. One can check that we gave dominating configurations for each claw of $G'$. So Claim \ref{claw3} is proved.

   \begin{claim}\label{claw4}
      $b(H)=1$ if and only if $b(G')=1$.
   \end{claim}

   Suppose that $b(G')=1$. In this case, we know that $cd$ is $\gamma$-critical in $G'$ (see the end of the arguments used to proved Claim \ref{planar9} in Theorem \ref{NPH_1-Bondage_planar}). Since $c,d$ are not vertices of an induced claw in $G'$, we have $cd\in H$. Hence from Claim \ref{claw2} it follows that $\gamma(H-cd)=\gamma(G'-cd)+2$, and therefore $cd$ is $\gamma$-critical in $H$.

   Now suppose that $b(G')\geq 2$. Suppose that $b(H)=1$. Let $xy$ be a $\gamma$-critical edge in $H$. From \ref{claw3}, no edge incident to an induced $H_v$ of $H$ is $\gamma$-critical. Therefore $xy\in E(G')$. Yet from Claim \ref{claw2} it follows that $\gamma(G'-xy)=\gamma(H-xy)-2$ and so $b(G')=1$, a contradiction. This proves Claim \ref{claw4}. \\

   From Claim \ref{claw2} and \ref{claw4}, by applying the operation $O_c$ iteratively until there is no induced claw, we obtain a planar claw-free graph $H$ with degree maximum $3$ such that $b(G')=b(H)$. This completes the proof.
\end{proof}

\begin{figure}[H]
   \centering
   \begin{subfigure}{.45\textwidth}
   \centering
   \includegraphics[width=15cm, height=3.5cm, keepaspectratio=true]{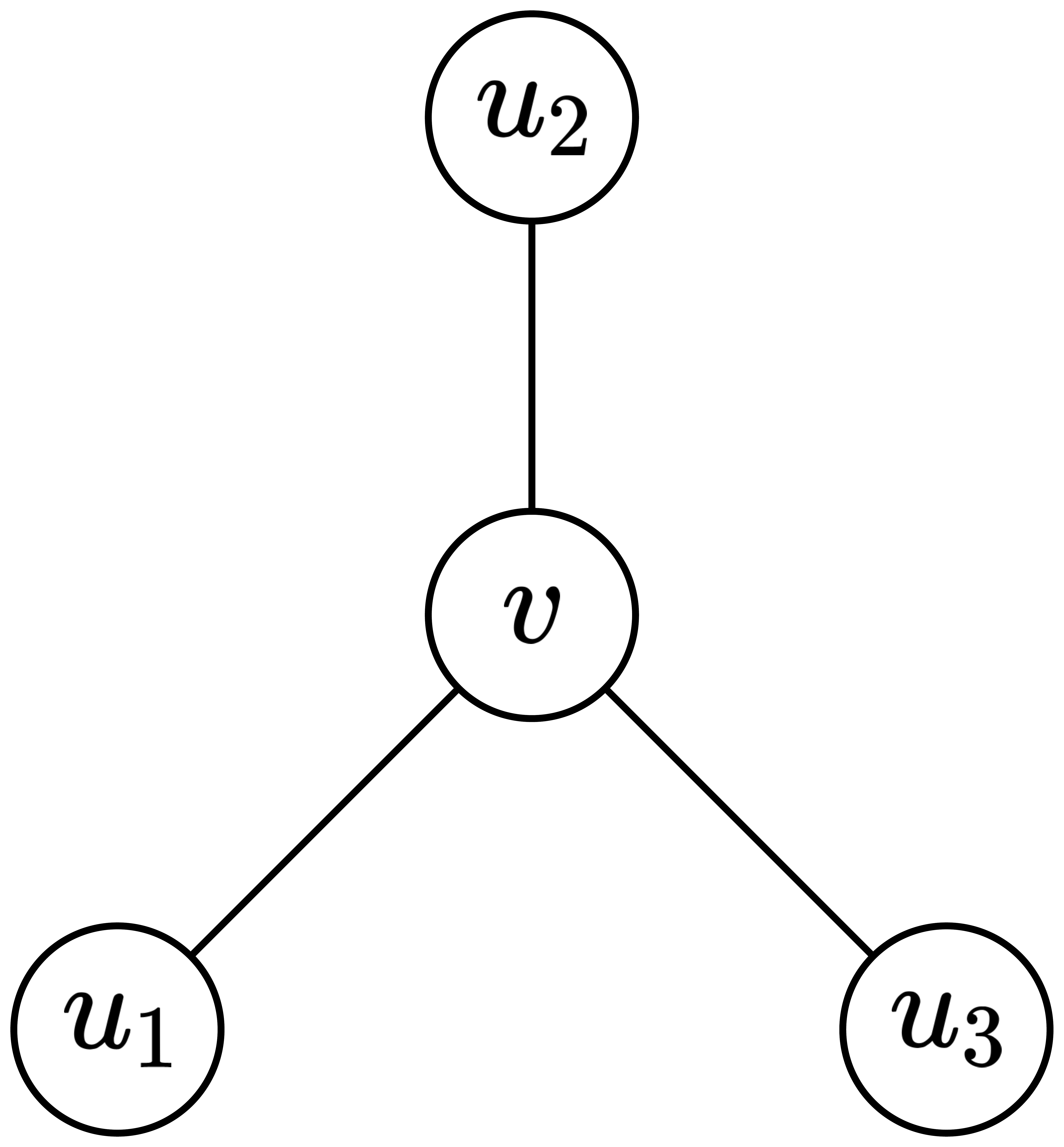}
   \end{subfigure}\hfill%
   \begin{subfigure}{.50\textwidth}
   \centering
   \includegraphics[width=15cm, height=6cm, keepaspectratio=true]{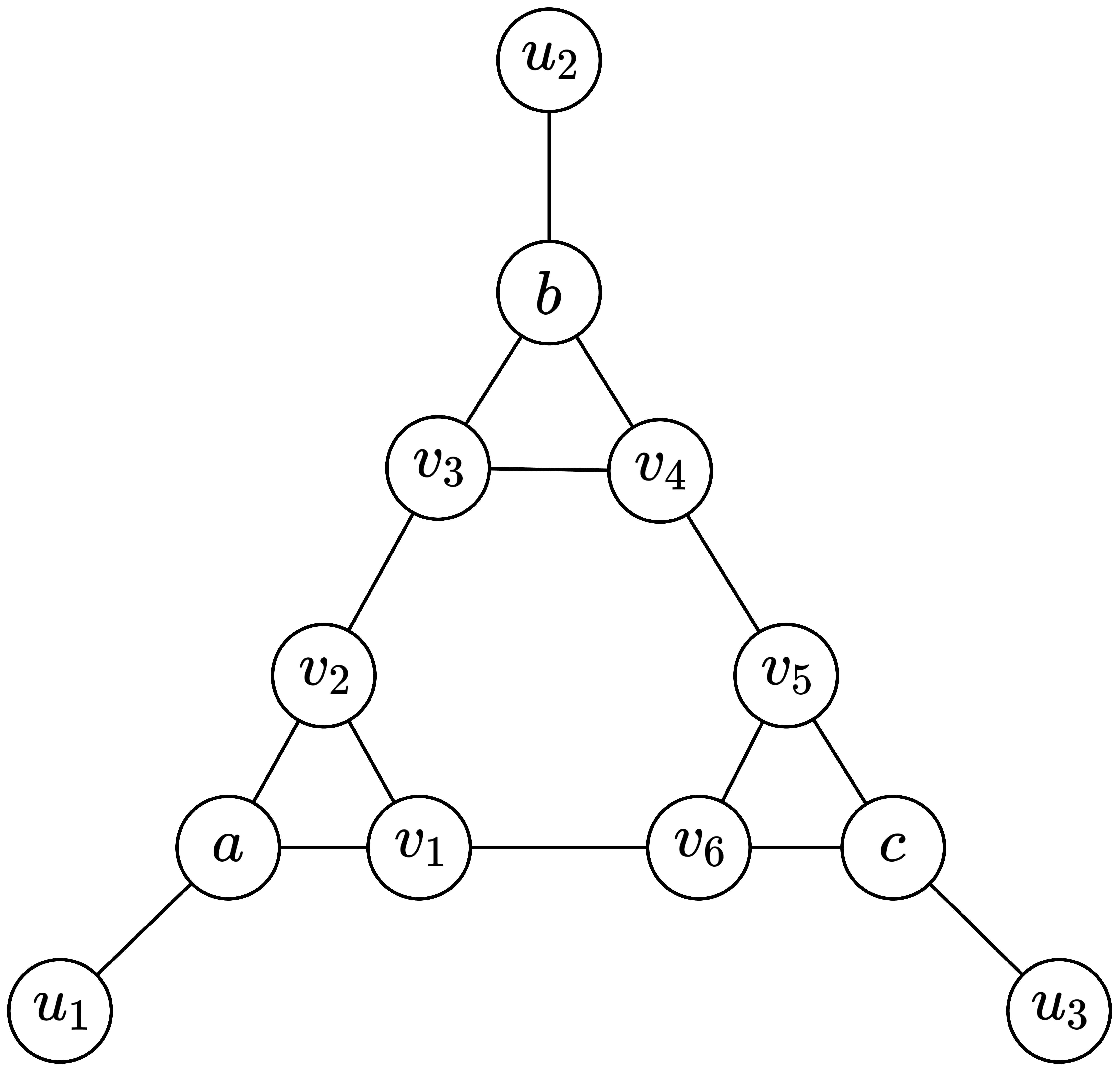}
   \end{subfigure}\hfill%
   \caption{On the left, an induced claw in $G'$, on the right, the corresponding component $H_v$.}
   \label{claw}
\end{figure}

\begin{theorem}\label{NPH_1-Bondage_planar_cubic}
   The $1$-\textsc{Bondage} is $\mathsf{NP}$-hard when restricted to $3$-regular planar graphs.
\end{theorem}
\begin{proof}
   We give a polynomial reduction from the $1$-\textsc{Bondage} which has been shown to be $\mathsf{NP}$-hard when restricted to subcubic planar graphs with bondage number at most $2$ in Theorem \ref{NPH_1-Bondage_planar}. From any connected planar graph $G=(V,E)$ such that $\Delta(G)\leq 3$, we construct a $3$-regular planar graph $G'=(V',E')$. Let $H_1$ and $H_2$ be the graph depicted in Figure \ref{deg1} and \ref{deg2}. One can easily check that the following statements are true:

   \begin{itemize}
      \item $\gamma(H_1)=\gamma(H_2)=2$;
      \item $\gamma(H_1 - \{u_1,u_2\})=\gamma(H_2-u)=2$;
      \item $b(H_1 - \{u_1,u_2\})=b(H_2)=2$;
      \item $u_1,u_2\in \gamma$-anticore$(H_1)$ and $u\in \gamma$-anticore$(H_2)$.
   \end{itemize}

   We describe the operation $O_1$ to increase the degree of a $1$-vertex to $3$. Let $v$ be a vertex of degree $1$ in $G$. Take a copy of $H_1$ and connect $v$ to $u_1$ and $u_2$. In the obtained graph $H$, we have $d_H(v)=3$. Since each vertex of $H_1$ in $H$ is of degree three, it follows that $H$ has one less $1$-vertex.

   We describe the operation $O_2$ to increase the degree of a $2$-vertex to $3$. Let $v$ be a vertex of degree $2$ in $G$. Take a copy of $H_2$ and connect $v$ to $u$. In the obtained graph $H$, we have $d_H(v)=3$ and exactly two vertices of $H_2$ in $H$ are $1$-vertices while the other are $3$-vertices. It follows that $H$ has one less $2$-vertex.

   Let $H$ be a graph obtained from $O_i$, and let $H_i$ be the added induced subgraph, where $i\in\{1,2\}$.

   \begin{claim}\label{cubic1}
      $\gamma(H)=\gamma(G)+2$
   \end{claim}

   Since $\gamma(H_i)=2$, we have $\gamma(H)\leq \gamma(G)+2$. Let $S$ be a minimum dominating set of $H$. Suppose that $\vert S\vert < \gamma(G)+2$. From our previous statements, we have $\vert S\cap V(H_i)\vert \geq 2$. If $\vert S\cap V(H_i)\vert = 2$, then from our previous arguments $u_1,u_2\not\in S$ or $u\not\in S$. Therefore $S\setminus H_i$ is a dominating set of $G$ such that $\vert S\setminus H_i\vert < \gamma(G)$, a contradiction. Else $\vert S\cap V(H_i)\vert \geq 3$ and $\{v\}\cup S\setminus V(H_i)$ is a $\gamma$-set of $G$, a contradiction. This proves Claim \ref{cubic1}.

   \begin{claim}\label{cubic2}
      $b(H)=1$ if and only if $b(G)=1$.
   \end{claim}

   Let $A\subseteq E$ such that $\gamma(G-A)=\gamma(G)+1$. From claim \ref{cubic1} $\gamma(H-A)=\gamma(G-A)+2$, and it follows that $b(H)\leq b(G)\leq 2$. Suppose that $b(H)=1$. Assume to the contrary that $b(G)\geq 2$. Let $uv$ be a $\gamma$-critical edge of $H$. Let $H_i=H_1$. There exists $S$ a $\gamma$-set of $G$ where $v\in S$ or the edge incident with $v$ in $G$ is $\gamma$-critical. Since $\gamma(H_1 - \{u_1,u_2\})=2$ and $b(H_1- \{u_1,u_2\})=2$, one can check that no edge incident with $H_1$ is $\gamma$-critical in $H$. Hence $uv\in E$ but then $\gamma(H-uv)=\gamma(G-uv)+2$, a contradiction. Now let $H_i=H_2$. Let $S,S'$ be two $\gamma$-sets of $G,H$, respectively. Since $S\cup S'$ is a $\gamma$-set of $H$ and that $b(H_2)=2$, then $uv$ has no endpoint in $H_2$. Hence $uv\in E$ and therefore $\gamma(H-uv)=\gamma(G-uv)+2$, a contradiction. This proves Claim \ref{cubic2}. \\

   Therefore from Claim \ref{cubic1} and \ref{cubic2}, by applying $O_1$ or $O_2$ until there exists a vertex of degree one or two, we obtain a planar $3$-regular graph $G'$ such that $b(G')=b(G)$. This completes the proof.
\end{proof}

\begin{figure}[H]
   \begin{subfigure}{.50\textwidth}
   \centering
   \includegraphics[width=15cm, height=3cm, keepaspectratio=true]{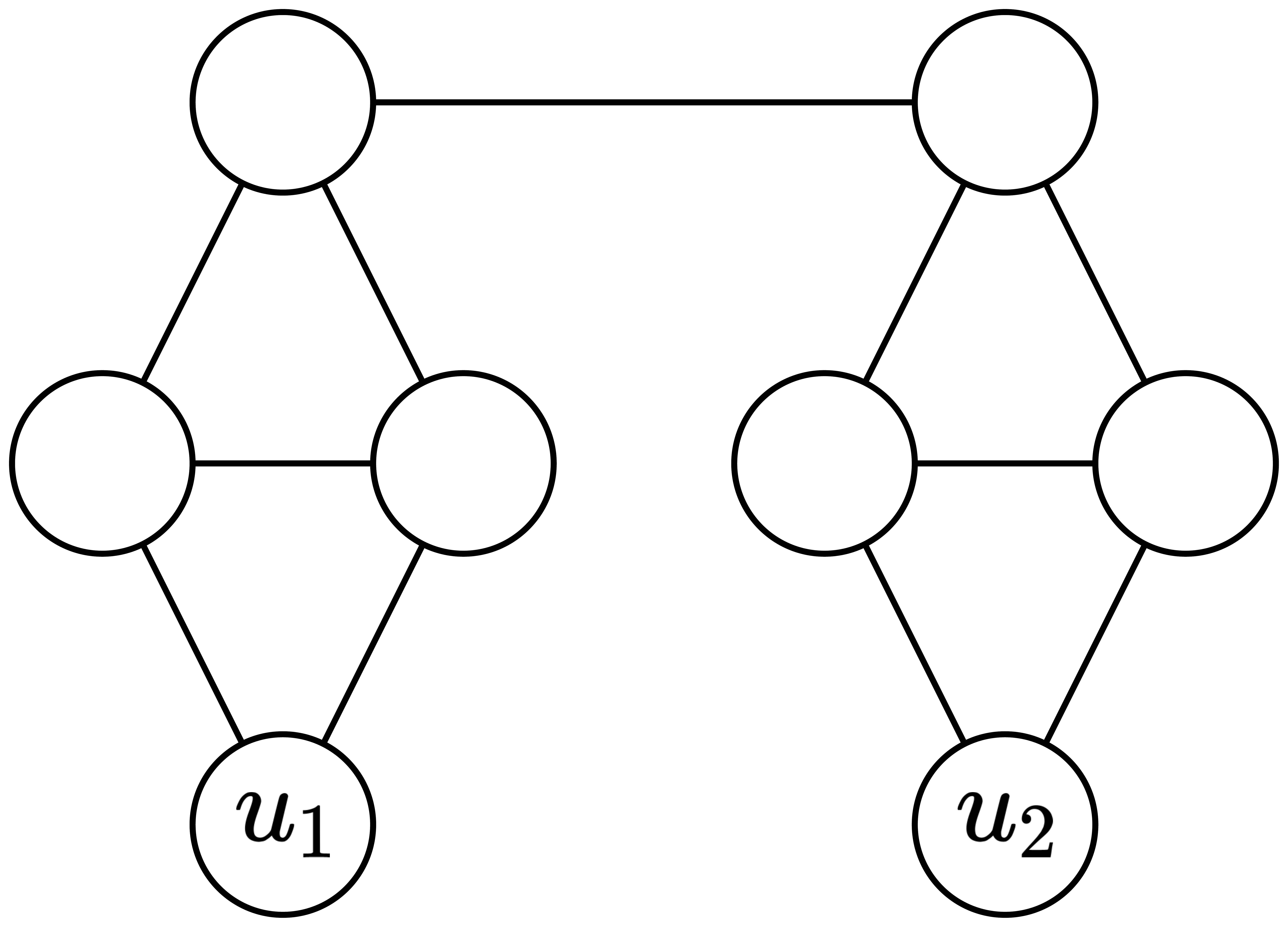}
   \caption{}
   \label{deg1}
   \end{subfigure}\hfill%
   \begin{subfigure}{.50\textwidth}
   \centering
   \includegraphics[width=15cm, height=3cm, keepaspectratio=true]{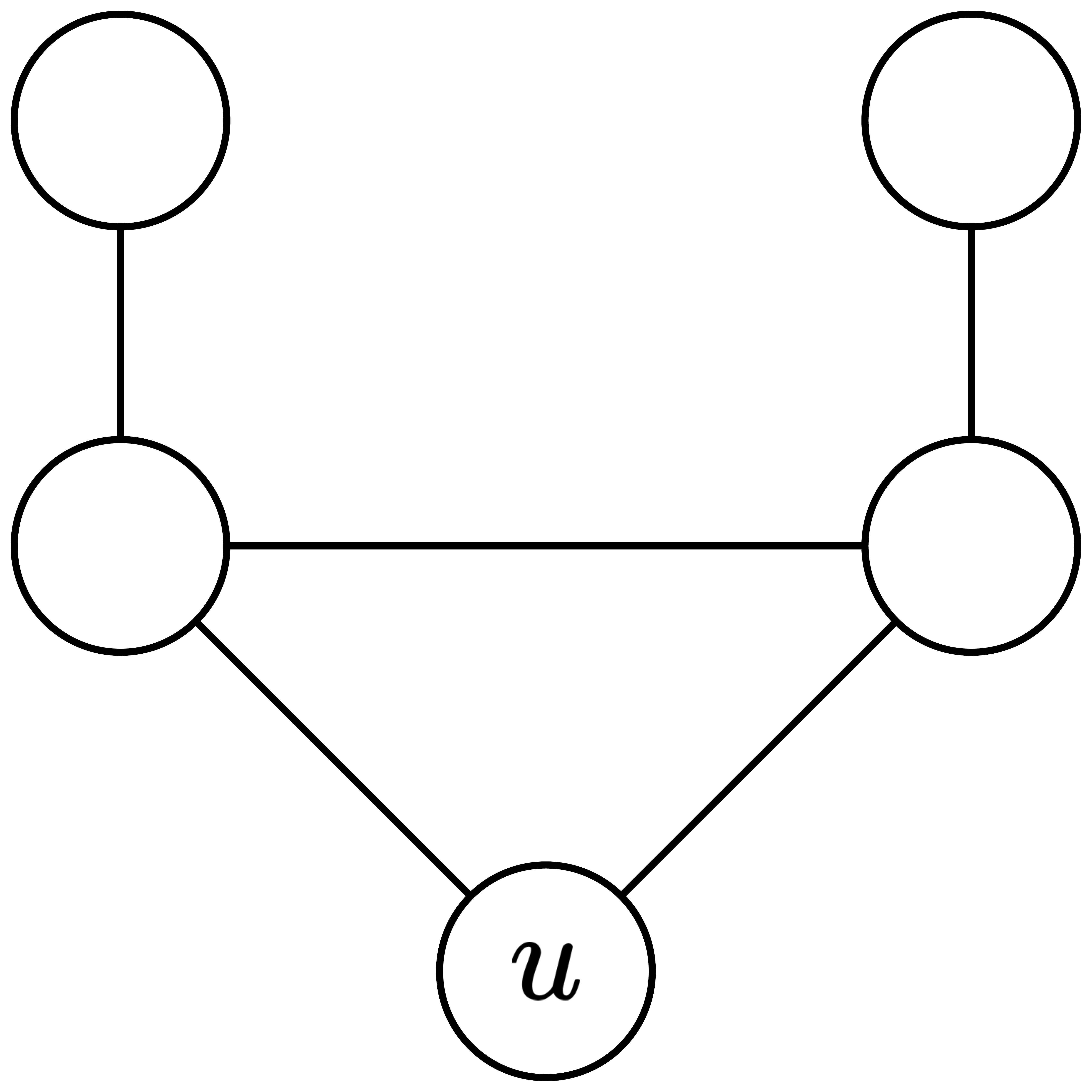}
   \caption{}
   \label{deg2}
   \end{subfigure}\hfill%
   \caption{(a) the graph $H_1$ to replace a $1$-vertex; (b) the graph $H_2$ to replace a $2$-vertex.}
\end{figure}

\begin{theorem}\label{NPH_1-Bondage_girth}
   For any fixed $k\geq 3$, $1$-\textsc{Bondage} is $\mathsf{NP}$-hard for bipartite planar graphs with degrees $2$ or $3$ and girth at least $k$.
\end{theorem}
\begin{proof}
   We give a polynomial reduction from $1$-\textsc{Bondage} which has been shown to be $\mathsf{NP}$-hard for planar cubic graphs with bondage number at most $2$ in Theorem \ref{NPH_1-Bondage_planar_cubic}. From a cubic planar graph $G=(V,E)$ such that $b(G)\leq 2$, and an integer $k$, we construct a planar bipartite graph $G'=(V',E')$ with $2\leq \delta(G')\leq \Delta(G')\leq 3$ and $g(G')\geq k$. \\

   The $3$-subdivision of an edge $uv$ is obtained by subdividing $uv$ into four edges and inserting three new vertices, say $u',w,v'$. Therefore $uv$ is replaced by a path of length five, that is, $P_5=u-u'-w-v'-v$. The $3$-subdivision has been used in \cite{Zver} (see Corollary 3) to prove the $\mathsf{NP}$-completeness of \textsc{Dominating Set} for bipartite graphs of girth at least $k$. It is proved that $\gamma(H)=\gamma(G)+1$, where $H$ is the graph obtained from the $3$-subdivision of an edge of $G$. Unfortunately, this operation does not always maintain the bondage number. Thereafter, we will introduced the operation $O_e$ to isolate small induced cycles for which edges can be $3$-subdivided without altering the bondage number.

   We claim the following:

   \begin{claim}\label{girth1}
      Let $H$ be the graph obtained from the $3$-subdivision of an edge $uv$ of $G$ such that $u,v\not\in \gamma$-anticore$(G)$. Then $\gamma(H)=\gamma(G)+1$ and $b(H)=1$ if and only if $b(G)=1$. As a consequence $u',w,v'\not\in \gamma$-anticore$(H)$.
   \end{claim}

   Since $\gamma(H)=\gamma(G)+1$ has already been proved in Theorem \cite{Zver}, we focus on the following to prove Claim \ref{girth1}:

   \begin{claim}\label{girth2}
      $b(H)=1$ if and only if $b(G)=1$.
   \end{claim}

   Suppose that $b(G)=1$. Let $st$ be a $\gamma$-critical edge of $G$. If $st\neq uv$, then $\gamma(H-st)=\gamma(G-st)+1$. Hence $st$ is $\gamma$-critical in $H$ and $b(H)=b(G)=1$. Else $st=uv$. Suppose that $uu'$ is not $\gamma$-critical in $H$. Let $S'$ be a $\gamma$-set of $H-uv$. It follows that $(N_G[u]\setminus \{v\})\cap S' \neq \emptyset$, and we can assume that $w\in S'$ since $u'$ is a leaf in $H-uv$. Since $v'$ is dominated by $w$ in $S'$, and that $d_H(v')=2$, we can assume that $(N_G[v]\setminus \{u\})\cap S'\neq \emptyset$. Therefore $S'\setminus \{w\}$ is a $\gamma$-set of $G$, a contradiction. The case $vv'$ is symmetric. So $uu'$ and $vv'$ are $\gamma$-critical in $H$ and $b(H)=b(G)=1$. \\

   Now suppose that $b(G)\geq 2$. Suppose that $b(H)=1$. Let $st$ be a $\gamma$-critical edge in $H$. If $st$ is an edge of $G$, then $\gamma(H-st)=\gamma(G-st)+1$, a contradiction. Hence $st\in \{uu',vv',u'w,v'w\}$. From a $\gamma$-set $S$ of $G-uv$, we have $S'=S\cup \{w\}$ a $\gamma$-set of $H$. Then from Remark \ref{noedgecritical} and \ref{novertexcritical}, $uu'$ and $vv'$ are not $\gamma$-critical in $H$. Moreover, there exists $S_1$ and $S_2$, two $\gamma$-sets of $G$ such that $u\in S_1$ and $v\in S_2$. Note that $S_1=S_2$ is possible. Then $S_1'=S\cup \{v'\}$ and $S_2'=S\cup \{u'\}$ are two $\gamma$-sets of $H$. Yet from Remark \ref{noedgecritical}, $u'w$ and $v'w$ are not $\gamma$-critical in $H$. So $b(H)\geq 2$. Note that the three new vertices $u',w,v'$, arising from the $3$-subdivision of $uv$, are not $\gamma$-anticores in the constructed graph $H$ since both $u$ and $v$ are not $\gamma$-anticores in $G$. This proves Claim \ref{girth2}. This proves Claim \ref{girth2} and therefore Claim \ref{girth1} is proved. \\


   Before introducing the operation $O_e$, we underline some properties of the graph used in this operation. Let $H_e$ be the graph represented in Figure \ref{he}. We leave to the reader the task of verifying that $\gamma(H_e)=8$ and $\gamma(H_e\setminus U)=8$, where $U\subseteq \{x,y\}$. We highlight six minimum dominating sets of $H_e$, that are:

   \begin{itemize}
      \item $A=\{a_1,a_3,b_1,b_4,d_2,e_2,f_2,g_2\}$;
      \item $B=\{a_1,a_3,b_1,b_4,d_3,e_3,f_3,g_3\}$;
      \item $C=\{a_1,a_3,b_2,c_3,d_3,e_1,f_1,g_2\}$;
      \item $D=\{a_2,b_1,b_4$, $c_1,d_4,e_3,f_2,g_4\}$;
      \item $E=\{a_1,a_3,b_3,c_4,d_3$, $e_2,f_4,g_1\}$;
      \item $F=\{a_1,b_2,b_4,c_2,d_1,e_4,f_3,g_4\}$.
   \end{itemize}

   From the $\gamma$-sets $A,B,C,D,E$ and Remark \ref{noedgecritical} and \ref{novertexcritical}, the graph $H_e$ admits no edge that is $\gamma$-critical and therefore $b(H_e)\geq 2$. Note that every vertex of $H_e\setminus \{x,y\}$ belongs to at least one $\gamma$-set. Then one can check that $\gamma$-anticore$(H_e)=\{x,y\}$. \\

   We describe the operation $O_e$. Select an edge $e=uv$ of $G$. Replace the edge $uv$ by the gadget $H_e$ such that $ux$ and $vy$ are two edges.
   We claim the following:

   \begin{claim}\label{girth3}
      Let $H$ be a graph obtained from the operation $O_e$ of an edge $e$ of $G$. Then $\gamma(H) = \gamma(G) + 8$ and $b(H)=1$ if and only if $b(G)=1$.
   \end{claim}

   First we focus on the following:

   \begin{claim}\label{girth4}
      $\gamma(H) = \gamma(G) + 8$.
   \end{claim}

   Let $S$ be a $\gamma$-set of $G$. From $S$ we construct $S'$ a dominating set of $H$. Let $S'=S$. If $u,v\not\in S$ or $u,v\in S$, then we set $S'\leftarrow S'\cup A$. Since $A$ is a $\gamma$-set of $H_e$, $S'$ dominates $H$. If $u\in S$ and $v\not\in S$, then we set $S'\leftarrow S'\cup \{a_1,b_2,c_3,d_2,e_1,f_1,g_2,y\}$. Note that $u$ dominates $x$ and $y$ dominates $v$. Therefore $S'$ is a dominating set of $H$. The case $u\not\in S$ and $v\in S$ is symmetric. Hence $\gamma(H)\leq \gamma(G)+8$.

   Now let $S'$ be a $\gamma$-set of $H$. From $S'$ we construct $S$ a $\gamma$-set of $G$. Let $S=S'\cap V(G)$. Since $\gamma(H_e\setminus U)=8$, where $U\subseteq \{x,y\}$, then $\vert S'\cap V(H_e)\vert \geq 8$. Hence $\vert S\vert \leq \gamma(G)$. If $\vert S\vert \leq \gamma(G)-2$, then $S\cup \{u\}$ is a $\gamma$-set of $G$, a contradiction. Suppose that $\vert S\vert = \gamma(G)-1$. If $S$ dominates $G$, then $\vert S\vert < \gamma(G)$, a contradiction. If $N_G[u]\cap S\neq \emptyset$ and $N_G[v]\cap S\neq \emptyset$, then $S$ is a dominating $G$ such that $\vert S\vert < \gamma(G)$, a contradiction. So $u,v\not\in S$. Hence $S$ does not dominate $H_e\cup U$ in $H$, where $U\subseteq \{u,v\}$. Since  $x,y\in \gamma$-anticore$(H_e)$ and $\gamma(H_e)=8$, we have $\vert S'\cap H_e\vert \geq 9$. So $\gamma(H)=\gamma(G)+8$. This proves Claim \ref{girth4}. \\

   To prove Claim \ref{girth3}, it remains to show the following:

   \begin{claim}\label{girth5}
      $b(H)=1$ if and only if $b(G)=1$.
   \end{claim}

   First suppose that $b(G)=1$. We may assume that $b(H)\geq 2$. Let $st$, $st\neq uv$, be a $\gamma$-critical edge of $G$. From Claim \ref{girth4} $\gamma(H-st)=\gamma(G-st)+8$, and therefore $b(H)=1$, a contradiction. Hence $uv$ is the unique $\gamma$-critical edge of $G$. Suppose that $ux$ is not $\gamma$-critical in $H$. From Claim \ref{girth4} $\gamma(H-ux)=\gamma(G)+8$. Let $S'$ be a $\gamma$-set of $H-ux$, and $S=S'\cap V(G)$. Note that $S\cap N_G[u]\neq \emptyset$. Since $\gamma(H_e\setminus U)=8$, where $U\subseteq \{x,y\}$, then $\vert S'\cap V(H_e)\vert \geq 8$. Hence $\vert S\vert \leq \gamma(G)$. If $S$ dominates $G-uv$, then $\vert S\vert < \gamma(G-uv)$, a contradiction. So $S\cap N_G[v]=\emptyset$, and therefore $y \in S'$. Yet from our previous arguments $\gamma(H_e)=8$ and $y\in \gamma$-anticore$(H_e)$. Therefore $\vert S' \cap V(H_e)\vert \geq 9$. Thus $\vert S\vert \leq \gamma(G)-1$. Then $S$ is dominating $G-u$ and $S\cup \{v\}$ is dominating $G-uv$. But $\vert S\cup \{v\}\vert < \gamma(G-uv)$, is a contradiction. So if $b(G)=1$, then $b(H)=1$.

   Now suppose that $b(G)\geq 2$. We may assume that $b(H)=1$. Let $st$ be a $\gamma$-critical edge in $H$. Let $S$ be a $\gamma$-set of $G-uv$. Note that $S'=S\cup A$ is a $\gamma$-set of $H-ux-vy$. Since $H_e$ is an isolated component in $H-ux-vy$ and that $b(H_e)\geq 2$, then $st$ is not an edge with an endpoint in $H_e$. Thus $st\in E(G)$ and then $\gamma(H-st)=\gamma(G-st) + 8$, a contradiction. So $b(H) \geq 2$. This proves Claim \ref{girth5}. Then Claim \ref{girth3} follows from Claim \ref{girth4} and \ref{girth5}.


   \begin{claim}\label{girth6}
      Let $H$ be the graph obtained from the $O_e$ operation of an edge $e=uv$ of $G$. If $uv$ if not $\gamma$-critical in $G$, then no vertices in $H_e\setminus \{x,y\}$ is in $\gamma$-anticore$(H)$.
   \end{claim}

   Let $S$ be a $\gamma$-set of $G-uv$. Then $S\cup X$, where $X\in \{A,B,C,D,E,F\}$, is a $\gamma$-set of $H$. So the vertices of $H_e\setminus \{x,y\}$ are not in $\gamma$-anticore$(H)$. this proves Claim \ref{girth6}. \\

   Last, we want to apply the $3$-subdivision on the subgraph $H_e$ so that it becomes bipartite, with girth at least $k$, and such that every induced path between $x$ and $y$ is odd. Let $H$ be the graph obtained from the operation $O_e$ of an edge $uv$ of the graph $G$. From Claim \ref{girth3}, we have $b(H)=1$ if and only if $b(G)=1$. From Claim \ref{girth1} and \ref{girth6}, we can $3$-subdivide any edge of $H_e\setminus \{x,y\}$ such that the resulting graph has bondage number $1$ if and only if $H$ has bondage number $1$.  Hence we can $3$-subdivide the bold and dotted edges of $H_e\setminus \{x,y\}$, as represented in Figure \ref{he_bip}. Each bold edge will be $3$-subdivided iteratively so that the path between its endpoints is even, while each dotted edge will be $3$-subdivided iteratively so that the path between its endpoints is odd. Multiples operations of $3$-subdivision will be applied until the modified graph $H_e$ has girth at least $k$. Then it remains to show that the arising graph $H_e'$ is bipartite and that there is no even path between $x$ and $y$. See this $2$-coloring of $H_e$ in Figure \ref{he_bip}. Then one can check the following:

   \begin{itemize}
      \item the vertices connected by an edge or an even path (bold edge) have distinct colors;
      \item the vertices connected by an odd path (dotted edge) have the same color;
      \item $x$ and $y$ have the same color and the colored paths between $x$ and $y$ alternate.
   \end{itemize}

   It follows that $H_e'$ is a planar, subcubic, bipartite, and has girth at least $k$. \\

   Applying iteratively the operation $H_e$ and the $3$-subdivision as done above, we obtain a graph $H$ that is planar, subcubic, bipartite, with girth at least $k$, and such that $b(H)=1$ if and only if $b(G)=1$. This completes the proof.
\end{proof}

\begin{figure}[H]
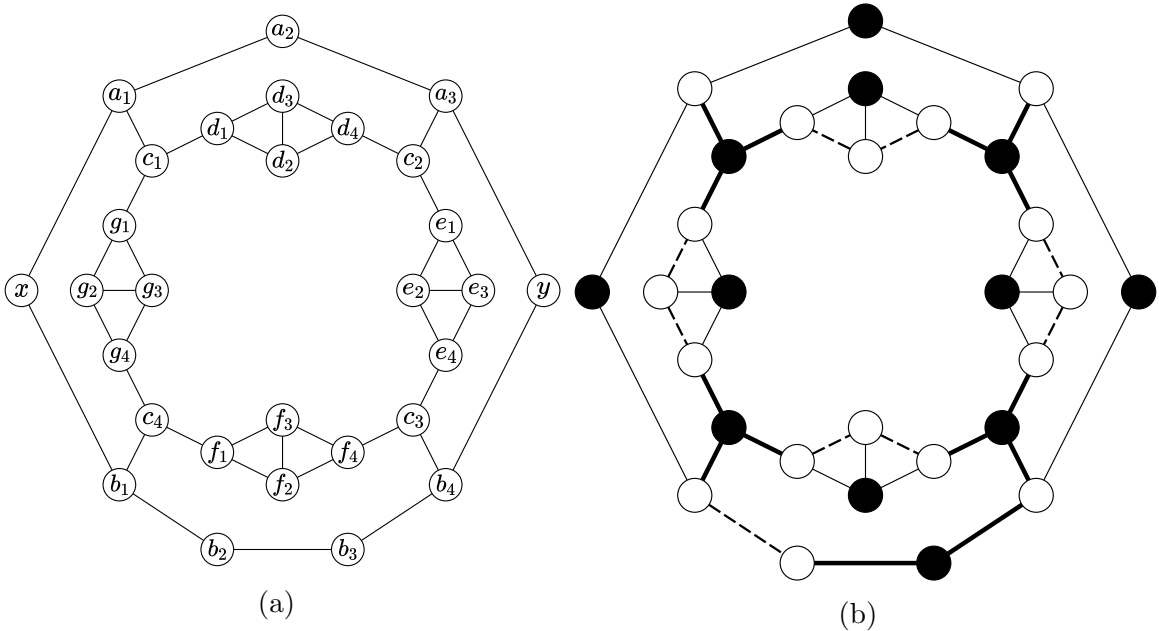

   \centering
   \begin{subfigure}{.48\textwidth}
   \centering
   \includegraphics[width=15cm, height=7.4cm, keepaspectratio=true]{He}
   \caption{}
   \label{he}
   \end{subfigure}\hfill%
   \begin{subfigure}{.50\textwidth}
   \centering
   \includegraphics[width=15cm, height=7.7cm, keepaspectratio=true]{he_bip}
   \caption{}
   \label{he_bip}
   \end{subfigure}\hfill%
   \caption{(a) The graph $H_e$ that corresponds to an edge; (b) $3$-subdivisions of the dotted and bold edges of the graph $H_e$, so that the resulting graph is bipartite with girth at least $k$.}
   \label{gadgets_He}
\end{figure}


In \cite{FischRau} (Theorem 3.5), it is proved that for any planar graph $G$ of girth at least $8$, we have $b(G)\leq 3$. Hence the answer to $3$-\textsc{Bondage} is always positive. Yet for these graphs, we may ask if we could find in polynomial time, a set of three edges $E'$ such that $\gamma(G-E')>\gamma(G)$. It has been shown in \cite{Zver} that \textsc{Dominating Set} is $\mathsf{NP}$-complete (see Corollary 3) for bipartite graphs of girth at least $k$, for any fixed $k\geq 3$. Hence we cannot find such a set of edges with brute force, that is, testing for all sets $E'$ of three edges if $\gamma(G-E')>\gamma(G)$. Fortunately we can reuse the arguments of \cite{FischRau} to find efficiently a set of three edges that is $\gamma$-critical.

\begin{proposition}
   Let $G=(V,E)$ be a planar graph of girth at least $8$. Then we can find $E'\subseteq E$, where $\vert E'\vert = 3$, such that $\gamma(G-E')>\gamma(G)$ in polynomial time.
\end{proposition}
\begin{proof}
   The proof in \cite{FischRau} (Theorem 3.5) to show that $b(G)\leq 3$ revolves around the existence of two vertices $u,v$ such that either:


   \begin{enumerate}
      \setlength\itemsep{.01em}
      \item[(1)] $uv\in E$ and $d(u) + d(v) \leq 4$; or
      \item[(2)] $d(u,v)\leq 2$ and $d(u) + d(v) \leq 3$.
   \end{enumerate}

   It is shown that such vertices always exist. Therefore for both cases, we expose a method to find a $\gamma$-critical set of edges of size at most three. If there is $u,v\in V$ as described in (1), then we take $E'$ as the set of edges incident with $u,v$. Remark that $\vert E'\vert \leq 3$. Let $S$ be a $\gamma$-set of $G-E'$. Since both vertices are isolated in $G-E'$ it follows that $u,v\in S$. Then $S\setminus \{u\}$ is a dominating set of $G$ and therefore $\gamma(G-E')>\gamma(G)$.

   Now suppose that there are $u,v\in V$ as described in (2). W.l.o.g.\ $d(v)=1$, $d(u)\leq 2$. Let $w\in N(u)\cap N(v)$. Then, we take $E'$ as the set of edges incident to $u$ and the edge $vw$. Remark that $\vert E'\vert \leq 3$ and that $u$ and $v$ are isolated $G-E'$. It follows that there is a $\gamma$-set $S$ of $G-E'$ such that $u,v\in S$. Therefore $(S\setminus \{u,v\})\cup \{w\}$ is a dominating set of $G$ and so $\gamma(G-E')>\gamma(G)$.

   Since we can test if there are two vertices that satisfy either (1) or (2) in polynomial time, this completes the proof.
\end{proof}

We carry on with a general approach for the classes of graphs for which $d$-\textsc{Bondage} can be solved in polynomial time. But first, consider the following operation: given $G=(V,E)$ and a set of vertices $A\subseteq V$, we construct $G_A+=(V',E')$ from a copy of $G$ as follows: for each $v\in A$ we add exactly one vertex $u$, $u\not\in V$, such that $uv$ is an edge. We show the following:

\begin{lemma}\label{dom_subset}
   Let $G=(V,E)$ be a graph and let $A\subseteq V$. There is a minimum dominating set $S$ of $G$, such that $A\subseteq S$, if and only $\gamma(G_A+)=\gamma(G)$.
\end{lemma}
\begin{proof}
   If there is a minimum dominating set $S$ of $G$ such that $A\subseteq S$, then $S$ is dominating set of $G_A+$. Now we can assume that $A\not\subseteq S$ for every minimum dominating set $S$ of $G$. Let $D$ be a minimum dominating set of $G_A+$. We may assume that $A\subseteq D$, since each vertex of $A$ has a leaf as neighbor in $G_A+$. Therefore $D$ is a dominating set of $G$ but $D$ is not a $\gamma$-set of $G$, and so $\gamma(G_A+) > \gamma(G)$.
\end{proof}

We are ready to show the following result. Note that it is inspired by the work in \cite{Ries} for reducing the domination number via edge contractions.

\begin{proposition}
   For any fixed $d\geq 1$, $d$-\textsc{Bondage} can be solved in polynomial time for $\mathcal{C}$, if either
   \begin{enumerate}
      \setlength\itemsep{.01em}
      \item[(1)] $\mathcal{C}$ is closed under edge deletions and \textsc{Dominating Set} can be solved in polynomial time for $\mathcal{C}$; or
      \item[(2)] $\mathcal{C}$ is the class of $H$-free graphs, where $H$ is a fixed graph with $\delta(H)\geq 2$, and \textsc{Dominating Set} can be solved in polynomial time for $\mathcal{C}$; or
      \item[(3)] for every $G\in \mathcal{C}$, $\gamma(G)\leq k$, where $k$ is a fixed constant.
   \end{enumerate}
\end{proposition}
\begin{proof}
   Let $G=(V,E)\in \mathcal{C}$. We prove (1). When $\mathcal{C}$ is closed under edge deletions, then we can compute in polynomial time the domination number of $G-E'$, for any set of edges $E'\subseteq E$, such that $\vert E'\vert \leq d$. Hence $d$-\textsc{Bondage} can be solved in polynomial time for $\mathcal{C}$. \smallskip

   We prove (2). Let $E'\subseteq E$, where $\vert E'\vert \leq d$, and let $G'=G-E'$. We define as $W$ the set of vertices covered by $E'$, that is, $W=\{v\in e \mid e\in E'\}$. Note that $\vert W\vert \leq 2d$. Let $\mathcal{A}$ be the collection of all $A\subseteq N[W]$, $\vert A\vert \leq 2d$, that dominates the vertices of $W$ in $G'$. One can see that $W\in \mathcal{A}$, and that $\vert \mathcal{A} \vert \leq \Delta(G)^{2d} \leq n^{2d}$.

   We show that $\gamma(G-E')>\gamma(G)$ if and only if for every minimum dominating set $S$ of $G$, and for every $A\in \mathcal{A}$, we have $A\not\subseteq S$.
   First, suppose that $\gamma(G-E')>\gamma(G)$. For contradiction, we assume that there is a minimum dominating set $S$ of $G$, and $A\in \mathcal{A}$, such that $A\subseteq S$. Therefore $S$ dominates $W$ in $G'$. Since $S$ dominates $V\setminus W$ in $G'$, it follows that $S$ is a dominating set of $G'$, a contradiction.
   Second, suppose that for every minimum dominating set $S$ of $G$, and for every $A\in \mathcal{A}$, we have $A\not\subseteq S$. For contradiction, we assume that $\gamma(G-E')=\gamma(G)$. Let $S$ be a minimum dominating set of $G$. Let $D$ be a minimal subset of $S\cap N[W]$ such that $D$ dominates $W$ in $G'$. Since $D$ is minimal, for every $u\in D$, there is $v\in N_{G'}(u)\cap W$ such that $N_{G'}(v)\cap D=\{u\}$. Therefore $\vert D\vert \leq \vert W\vert \leq 2d$, and $D\in \mathcal{A}$, and $D\subseteq S$, a contradiction.



   So from Lemma \ref{dom_subset}, it follows that $\gamma(G-E')>\gamma(G)$ if and only if for every $A\in \mathcal{A}$, we have $\gamma(G_A+)> \gamma(G)$. Recall that $G_A+$ is a copy of $G$ plus $\vert A\vert$ leaves. Therefore if $G$ is $H$-free and $\delta(H)\geq 2$, it follows that $G_A+$ is $H$-free. Hence we can compute $\gamma(G)$ and $\gamma(G_A+)$ in polynomial time. Since there is at most $\mathcal{O}(n^{2d})$ sets $E'$, and their corresponding collections $\mathcal{A}$ are of size at most $\mathcal{O}(n^{2d})$, it follows that we can check if there exists such a set $E'$, with $\gamma(G-E')>\gamma(G)$, in polynomial time. So $d$-\textsc{Bondage} can be solved in polynomial time for $\mathcal{C}$. \smallskip

   We prove (3). Since $k$ is fixed, we can test for every set $S\subseteq V$, where $\vert S\vert \leq k$, if $S$ is a dominating set of $G$ in polynomial time. It follows that we can compute the list of minimum dominating set of $G$ in polynomial  time. Then for any set of edges $E'\subseteq E$, where $\vert E'\vert \leq d$, we define $\mathcal{A}$ as in the proof of item (2). Recall that $\vert \mathcal{A}\vert \leq \mathcal{O}(n^{2d})$. To test if $\gamma(G-E')>\gamma(G)$, we make use of Lemma \ref{dom_subset}. Therefore we check if for every $E'\subseteq E$, such that $\vert E'\vert \leq d$, and for every $A\in \mathcal{A}$ if $\gamma(G_A+)\neq \gamma(G)$. So $d$-\textsc{Bondage} can be solved in polynomial time.
\end{proof}

To conclude this section, we give an answer to the following question raised in \cite{HuSohn}, that is, does \textsc{Bondage} belongs to $\mathsf{NP}$ ? We remark that if \textsc{Bondage} is in $\mathsf{NP}$, then there must exists a certificate, that we can test in polynomial time, with the following characteristics: given a graph $G$ and $E'\subseteq E(G)$, is $\gamma(G-E')>\gamma(G)$ ? Therefore we should be able to test if a given edge is $\gamma$-critical in $G$ in polynomial time. Yet we show that this is not possible, unless $\mathsf{P}=\mathsf{NP}$.

\begin{theorem}
   Given a graph $G$ and an edge $e\in E(G)$, there is no polynomial time algorithm deciding if $\gamma(G-e)=\gamma(G)+1$, unless $\mathsf{P}=\mathsf{NP}$.
\end{theorem}
\begin{proof}
   We define as $\gamma$-\textsc{Edge Deletion} the problem of deciding, given a graph $G$ and an edge $e\in E(G)$, if $\gamma(G-e)=\gamma(G)+1$. We show that if $\gamma$-\textsc{Edge Deletion} can be solved in polynomial time, then we can solve \textsc{Dominating set} in polynomial time.

   Let $(G,k)$ be an instance for \textsc{Dominating Set}. We set $G'\leftarrow G$ and $p\leftarrow 0$. We describe the following procedure. We pick an edge $e\in E(G')$ and set $G'\leftarrow G'-e$. We solve $\gamma$-\textsc{Edge Deletion} with input $(G',e)$ in polynomial time. If $(G',e)$ is a \textsc{Yes-instance}, then we set $p\leftarrow p + 1$. If $V(G')$ is an independent set of $G'$, then we stop this procedure, else we continue. Note that at each step, the instance for $\gamma$-\textsc{Edge Deletion} has one less edge. Hence there is at most $\vert E(G)\vert$ steps for this procedure to end.

   At the end of the procedure, it follows that $\gamma(G')=\vert V(G)\vert$. We remark that $\gamma(G)+p=\gamma(G')$ since $p$ is incremented each time the domination number increase by one. Therefore $\gamma(G) = \vert V(G)\vert - p$, and it follows that $(G,k)$ is a \textsc{Yes-instance} of \textsc{Dominating Set} if and only if  $\vert V(G)\vert - p \leq k$. Therefore we can solve \textsc{Dominating Set} in polynomial time. Yet \textsc{Dominating Set} is known to be $\mathsf{NP}$-complete, see \cite{GareyJohnson}. This completes the proof.
\end{proof}

\begin{coro}
   \textsc{Bondage} is not in $\mathsf{NP}$, unless $\mathsf{P}=\mathsf{NP}$.
\end{coro}

\section{Conclusion and open problems}

We proved the $\mathsf{NP}$-hardness of $1$-\textsc{Bondage} for the following classes of graphs:
\begin{itemize}
   \item planar cubic graphs;
   \item planar claw-free graphs with maximum degree $3$;
   \item planar bipartite graphs with degrees $2$ or $3$ and girth at least $k$, for any fixed $k\geq 3$.
\end{itemize}

It is natural to ask for the complexity of the \textsc{Bondage} for bipartite cubic graphs of girth at least $k$, for any fixed $k\geq 3$. Yet from our construction, it seems not possible to remove vertices of degree two without introducing small cycles. Hence the study of $(C_4,...,C_k)$-free graphs would be of interest. Complexity results for $H$-free graphs, for fixed $H$, would also be of important values. Another important study concerns the complexity of the $d$-\textsc{Bondage} for fixed $d\geq 2$. To conclude, we want to raise the following problem: \\

\textbf{Problem}: Characterize the planar graphs of girth $8$ with bondage number $3$.

\begin{ack}
   The author express its gratitude to Fran\c cois Delbot, Christophe Picouleau and St\'ephane Rovedakis for carefully reading the manuscript, and for providing helpful comments.
\end{ack}

\end{document}